\documentclass[11pt,reqno]{amsart}
\usepackage{amsmath,amssymb,extarrows}
\usepackage{url}
\usepackage{tikz,enumerate}
\usepackage{diagbox}
\usepackage{appendix}
\usepackage{epic}
\usepackage{cite}

\usepackage{hyperref}
\usepackage{array}

\usepackage{booktabs}

\allowdisplaybreaks[4]

\newtheorem{theorem}{Theorem}
\newtheorem{corollary}[theorem]{Corollary}
\newtheorem{conj}[theorem]{Conjecture}
\newtheorem{lemma}[theorem]{Lemma}

\theoremstyle{definition}
\newtheorem{defn}{Definition}
\theoremstyle{remark}
\newtheorem{rem}{Remark}
\numberwithin{equation}{section}
\numberwithin{theorem}{section}
\numberwithin{defn}{section}

\begin{document}
\title[Parity of coefficients of mock theta functions]
 {Parity of coefficients of mock theta functions}

\author{Liuquan Wang}
\address{School of Mathematics and Statistics, Wuhan University, Wuhan 430072, Hubei, People's Republic of China}
\email{wanglq@whu.edu.cn;mathlqwang@163.com}




\begin{abstract}
We study the parity of coefficients of classical mock theta functions. Suppose $g$ is a formal power series with integer coefficients, and let $c(g;n)$ be the coefficient of $q^n$ in its series expansion. We say that $g$ is of parity type $(a,1-a)$ if $c(g;n)$ takes even values with probability $a$ for $n\geq 0$. We show that among the 44 classical mock theta functions, 21 of them are of parity type $(1,0)$. We further conjecture that 19 mock theta functions are of parity type $(\frac{1}{2},\frac{1}{2})$ and 4 functions are of parity type $(\frac{3}{4},\frac{1}{4})$. We also give characterizations of $n$ such that $c(g;n)$ is odd for the mock theta functions of parity type $(1,0)$.
\end{abstract}

\maketitle

\section{Introduction}\label{sec-intro}
In 1920, Ramanujan introduced mock theta functions in his last letter to Hardy. He gave a list of 17 mock theta functions and defined each function as a $q$-series in Eulerian form, and he associated with each mock theta function an order. His list contains mock theta functions of orders 3, 5 and 7. For example,
\begin{align}
f^{(3)}(q):=\sum_{n=0}^{\infty}\frac{q^{n^2}}{(-q;q)_{n}^2} \quad \text{\rm{and}} \quad  f_0^{(5)}(q):=\sum_{n=0}^{\infty}\frac{q^{n^2}}{(-q;q)_{n}}
\end{align}
are two mock theta functions of orders 3 and 5, respectively.
Here we follow \cite{Chen-Wang} to add a superscript $(n)$ to indicate that a mock theta function is of order $n$. We also adopt the customary $q$-series notation:
\begin{align}
&(a;q)_\infty =\prod\limits_{n=0}^\infty (1-aq^n), \quad |q|<1, \\
&(a;q)_n=\frac{(a;q)_\infty}{(aq^n;q)_\infty}, \quad n\in \mathbb{N},\\
&(a_1,a_2,\cdots, a_m;q)_n =(a_1;q)_n(a_2;q)_n\cdots (a_m;q)_n, \quad n\in \mathbb{N}\cup \{\infty \}.
\end{align}
Meanwhile, Ramanujan presented some identities satisfied by mock theta functions. He also recorded  identities for mock theta functions of orders 6 and 10 in his lost notebook \cite{lostnotebook}. After their appearance, there  have been numerous studies of mock theta functions.

In the past, people are mainly focused on studying  Appell-Lerch and Hecke-type series representations of mock theta functions as well as identities satisfied by them. Watson \cite{Watson,Watson-2} found Appell-Lerch series representations for mock theta functions of order 3. For instance, he \cite{Watson}  provided the following Appell-Lerch series representation:
\begin{align}
f^{(3)}(q)=\frac{2}{(q;q)_{\infty}}\sum_{n=-\infty}^{\infty}\frac{(-1)^nq^{\frac{3}{2}n^2+\frac{1}{2}n}}{1+q^n}. \label{intro-f(q)}
\end{align}
Andrews \cite{Andrews-TAMS} found  Hecke-type series representations for mock theta functions of orders 5 and 7 such as
\begin{align}
f_0^{(5)}(q)=\frac{1}{(q;q)_{\infty}}\sum_{n=0}^{\infty}\sum_{|j|\leq n}(-1)^jq^{\frac{5}{2}n^2+\frac{1}{2}n-j^2}(1-q^{4n+2}). \label{intro-f0}
\end{align}
More Appell-Lerch and Hecke-type series representations can be found  in the works of  Watson \cite{Watson,Watson-2}, Andrews \cite{Andrews-TAMS}, Andrews and Hickerson \cite{Andrews-Hickerson}, Berndt and Chan \cite{Berndt-Chan}, Choi \cite{Choi-1,Choi-2}, Cui, Gu and Hao \cite{CGH}, Garvan \cite{Garvan-2015,Garvan-arXiv}, Gordon and McIntosh \cite{Gordon-McIntosh}, Hickerson \cite{Hickerson}, Mortenson \cite{Mortenson-2013}, Srivastava \cite{Srivastava} and Zwegers \cite{Zwegers-Rama}. Around 2002, based on Appell-Lerch and Hecke-type series representations of mock theta functions, Zwegers \cite{Zwegers} successfully explained the modular properties of mock theta functions. For a more comprehensive historical background on mock theta functions, see the survey of Gordon and McIntosh \cite{Gordon-McIntosh-Survey}, the paper of Hickerson and Mortenson \cite{Hickerson-Mortenson} or the recent book of Andrews and Berndt \cite{lost-notebook5}. We remark that in a recent work \cite{Chen-Wang}, Chen and the author provided a unified method for establishing Appell-Lerch and Hecke-type series representations for mock theta functions of orders 2, 3, 5, 6 and 8.

Along the process of understanding mock theta functions, people are also interested in arithmetic properties of their coefficients. This was again motivated by the work of Ramanujan, who found that the partition function $p(n)$, enumerating the number of partitions of $n$, satisfies beautiful congruences such as
\begin{align}
p(5n+4) &\equiv 0 \pmod{5}, \\
p(7n+5) &\equiv 0 \pmod{7}, \\
p(11n+6) &\equiv 0 \pmod{11}.
\end{align}
Since then  a rich theory on arithmetic properties of  partitions has been developed by many mathematicians. In contrast, arithmetic properties of coefficients of mock theta functions were less understood. Nevertheless, there have been a number of works which provide congruences satisfied by mock theta functions.

For any formal power series
\begin{align}\label{formal-eq}
g(q)=\sum_{n=0}^ \infty c(n)q^n,
\end{align}
we use $c(g;n)=c(n)$ to denote the coefficient of $q^n$ in the series expansion of $g(q)$. Garthwaite and Penniston \cite{GP} showed that $c(\omega^{(3)};n)$ satisfies infinitely many Ramanujan-type congruences.  Waldherr \cite{Waldherr} gave the first explicit examples of such congruences:
\begin{align}
c(\omega^{(3)};40n+27)\equiv c(\omega^{(3)};40n+35) \equiv 0 \pmod{5}.
\end{align}
Bruinier and Ono \cite{BO} proved some nice identities and congruences modulo 512 for $c(\omega^{(3)};n)$.
 Andrews, Passary, Sellers and Yee \ \cite{APSY} found more congruences like
\begin{align*}
c(\omega^{(3)};8n+3) &\equiv 0 \pmod{4}, \\
c(\omega^{(3)};8n+5) &\equiv 0 \pmod{8}.
\end{align*}
The author \cite{Wang2017} found some congruence for $\nu^{(3)}(q)$ such as
\begin{align}
c(\nu^{(3)};4n+2) &\equiv 0 \pmod{2}.
\end{align}
Several congruences for $c(\omega^{(3)};n)$ and $c(\nu^{(3)};n)$ modulo 11 were also presented in \cite{Wang2017}.  For more congruences satisfied by coefficients of mock theta functions, see the works of Berg et al. \cite{BCGKMW}, Brietzke, Silva and Sellers \cite{BSS}, Chan and Mao \cite{Chan-Mao}, Chern and Wang \cite{Chern-Wang},  Garthwaite \cite{Garthwaite}, Lin \cite{Lin}, Mao \cite{Mao,Mao-BAMS}, Qu, Wang and Yao \cite{Qu-Wang-Yao} and Xia \cite{Xia}, for example.

For any formal power series $g(q)$ as in \eqref{formal-eq}, we define its type modulo a positive integer as follows.
\begin{defn}
Given a positive integer $m$, if the following limits exist
\begin{align}
\lim\limits_{N\rightarrow \infty} \frac{\#\{0\leq n< N: c(g;n)\equiv i\,\, \text{ \rm{(mod $m$)}}\}}{N} =\alpha_i, \quad i=0,1,\dots, m-1.
\end{align}
Then we say that $g$ and the sequence $\{c(g;n): n\geq 0\}$ are \textit{of type $(\alpha_0,\alpha_1,\dots, \alpha_{m-1})$ modulo $m$}. In particular, when $m=2$, we also say that $g$ and the sequence $\{c(g;n): n\geq 0\}$ are of \textit{parity type} $(\alpha_0,\alpha_1)$.
\end{defn}
It is clear from definition that $\alpha_0+\alpha_1+\cdots +\alpha_{m-1}=1$.

In this paper, we will focus on the parity of coefficients of mock theta functions.  This is mainly motivated by research on the parity of the partition function $p(n)$. It has long been conjectured that $p(n)$ is of type $(\frac{1}{2},\frac{1}{2})$ modulo 2, i.e.,
\begin{align}\label{pn-half}
\#\{0\leq n\leq N: p(n)  \text{ is odd (even)}\} \sim \frac{1}{2}N.
\end{align}
By doing extensive computations, Parkin and Shanks \cite{Parkin-Shanks} provided numerical evidence indicating that \eqref{pn-half} is very likely to be true.
In 1959, Kolberg \cite{Kolberg} showed that $p(n)$ assumes either even or odd values infinitely often. Mirsky \cite{Mirsky} found the first quantitative result which states that
\begin{align}
\#\{0\leq n\leq N: p(n) \text{ is odd (even)}\} >\frac{\log\log N}{2\log 2}.
\end{align}
After then, this lower bound has been improved in a number of works. For example, In an appendix to \cite{NRS}, Serre proved that
\begin{align}
\lim_{N\rightarrow \infty} \frac{\#\{0\leq n\leq N: p(n) \text{ is even}\}}{\sqrt{N}} =\infty.
\end{align}
For a good historical account on this process, see \cite[Sec.\ 2.5]{Berndt-book}. To the best of our knowledge, so far the best known result for even values of $p(n)$ was given by Bella\" iche and Nicolas \cite{Nicolas}:
\begin{align}
\#\{0\leq n\leq N: p(n) \text{ is even}\} \geq 0.069\sqrt{N}\log\log N. \label{best-even}
\end{align}
While the best known results for odd values of $p(n)$ was given by Bella\" iche, Green and Soundararajan \cite{BGS}:
\begin{align}
\#\{0\leq n\leq N: p(n) \text{ is odd} \}\gg \frac{\sqrt{N}}{\log\log N}. \label{best-odd}
\end{align}

There have also been other research on the parity of certain coefficients defined by $q$-series in Eulerian form. For example, recall the famous Rogers-Ramanujan identities
\begin{align}
G(q):=\sum_{n=0}^\infty \frac{q^{n^2}}{(q;q)_n}=\frac{1}{(q;q^5)_\infty (q^4;q^5)_\infty}, \\
H(q):=\sum_{n=0}^\infty \frac{q^{n^2+n}}{(q;q)_n} =\frac{1}{(q^2;q^5)_\infty(q^3;q^5)_\infty}.
\end{align}
Gordon \cite{Gordon} proved that $c(G;n)$ is odd for $n$ odd if and only if $60n-1=p^{4a+1}m^2$ for prime $p$ and integer $m$ with $p\nmid m$. Similarly, he proved that $c(H;n)$ is odd for $n$ even if and only if $60n+11=p^{4a+1}m^2$ for prime $p$ and integer $m$ with $p\nmid m$. Based on these characterizations, Chen \cite{Chen} proved that for sufficiently large $N$,
\begin{align}
\#\{0\leq n\leq N: c(G;2n+1)\equiv 1\,\, \mathrm{ (mod \,\, 2)}\}
=\frac{2\pi^2}{5}\frac{N}{\log N}+O\left(\frac{N\log\log N}{\log^2N} \right),  \label{intro-Chen-G} \\
\#\{0\leq n\leq N: c(H;2n)\equiv 1 \,\, \mathrm{ (mod \,\, 2)}\}=\frac{2\pi^2}{5}\frac{N}{\log N}+O\left(\frac{N\log\log N}{\log^2N} \right). \label{intro-Chen-H}
\end{align}

Inspired by the above works, the main goal of this paper is to give a systematic study on the parity of coefficients of classical mock theta functions. Given a mock theta function $g(q)$, our aim is to find its type modulo 2 and give characterizations for $n$ such that $c(g;n)$ is odd.
Following the notations of mock theta functions in \cite{Chen-Wang}, we can state our main result as follows.
\begin{theorem}\label{thm-main}
The following $21$ mock theta functions
\begin{align*}
&A^{(2)}(q), B^{(2)}(q),  \psi^{(3)}(q), \omega^{(3)}(q), \nu^{(3)}(q),  \rho^{(3)}(q),  \psi_0^{(5)}(q), \psi_1^{(5)}(q), F_0^{(5)}(q), F_1^{(5)}(q), \rho^{(6)}(q), \\
&  \sigma^{(6)}(q),   \phi_{-}^{(6)}(q), \psi_{-}^{(6)}(q), T_0^{(8)}(q),   T_1^{(8)}(q), U_1^{(8)}(q),  V_0^{(8)}(q), V_1^{(8)}(q), \phi^{(10)}(q), \psi^{(10)}(q)
\end{align*}
are all of parity type $(1,0)$. Furthermore, for any mock theta function $g$ from the above list excluding $\rho^{(3)}(q)$, we can find all the $n$ such that $c(g;n)$ is odd.
\end{theorem}
The functions listed in Theorem \ref{thm-main} can be further divided into three groups. The first group consists of $B^{(2)}(q)$, $\omega^{(3)}(q)$, $\nu^{(3)}(q)$,  $\rho^{(6)}(q)$, $V_0^{(8)}(q)$, $\phi^{(10)}(q)$ and $\psi^{(10)}(q)$. For $g(q)$ in the first group, the coefficient $c(g;n)$ is odd if and only if $n=P(k)$ for some quadratic polynomial $P(x)\in \mathbb{Q}[x]$, and hence the number of odd values can be precisely counted. For example, for the second order mock theta function $B^{(2)}(q)$ (see \eqref{mock-2-B-defn}), we show that (see Theorem \ref{thm-mock-2-B}) the coefficient $c(B^{(2)};n)$ is odd if and only if $n=2k^2+2k$ for some $k\geq 0$. As a consequence,
\begin{align}\label{intro-mock-2-B}
\#\left\{0\leq n\leq N|c(B^{(2)};n)\equiv 1 \,\, (\mathrm{mod \,\, 2}) \right\}=\left\lfloor \frac{\sqrt{2N+1}+1}{2}\right\rfloor,
\end{align}
where $\lfloor x\rfloor$ denotes the integer part of a real number $x$.

The second group consists of $A^{(2)}(q)$,  $\psi^{(3)}(q)$,  $\psi_0^{(5)}(q)$, $\psi_1^{(5)}(q)$, $F_0^{(5)}(q)$, $F_1^{(5)}(q)$, $\sigma^{(6)}(q)$, $\phi_{-}^{(6)}(q)$, $\psi_{-}^{(6)}(q)$,  $T_0^{(8)}(q)$,   $T_1^{(8)}(q)$, $U_1^{(8)}(q)$ and $V_1^{(8)}(q)$.
For any mock theta function $g(q)$ from this group, the parity of $c(g;n)$ is similar to $c(G;2n+1)$ and $c(H;2n)$. For instance, for the second order mock theta function $A^{(2)}(q)$ (see \eqref{mock-2-A-defn}),
we prove that (see Theorem \ref{thm-mock-2-A}) the number $c(A^{(2)};n)$ is odd if and only if $8n-1=m^2p^{4a+1}$, where $p$  is a prime, $m$ is a positive integer and $p\nmid m$. Furthermore, we find that
\begin{align}\label{intro-mock-2-A-estimate}
\#\left\{0\leq n\leq N|c(A^{(2)};n)\equiv 1  (\mathrm{mod \,\, 2}) \right\}=\frac{\pi^2}{4}\frac{N}{\log N}+O\left(\frac{N}{\log^2N} \right).
\end{align}

The third group consists of the single function $\rho^{(3)}(q)$. We have not been able to find all $n$ such that $c(\rho^{(3)};n)$ is odd. Nevertheless, we will show that $c(\rho^{(3)};4n)$ is odd if and only if $n=2k(3k+1)$ for some integer $k$, and $c(\rho^{(3)};2n+1)$ is odd if and only if $6n+5=p^{4a+1}m^2$ for some prime $p$ and integer $m$ with $p\nmid m$. As such, the function $\rho^{(3)}(q)$ mixes the parity properties of functions in the first two groups.

For the remaining 23 classical mock theta functions, we have not succeeded in giving characterizations of their odd values.
Based on computations of their coefficients and known properties, we make the following conjecture.
\begin{conj}\label{conj-mock}
The following $19$ mock theta functions
\begin{align*}
&\mu^{(2)}(q), f^{(3)}(q), \phi^{(3)}(q), \chi^{(3)}(q), \phi_0^{(5)}(q), \phi_1^{(5)}(q), \chi_0^{(5)}(q),  \chi_1^{(5)}(q),    \phi^{(6)}(q),
\psi^{(6)}(q),  \\
& \gamma^{(6)}(q), \mathcal{F}_0^{(7)}(q),  \mathcal{F}_1^{(7)}(q), \mathcal{F}_2^{(7)}(q),    S_0^{(8)}(q), S_1^{(8)}(q), U_0^{(8)}(q),  X^{(10)}(q), \chi^{(10)}(q).
\end{align*}
are all of parity type $(\frac{1}{2},\frac{1}{2})$.

The $4$ functions  $f_0^{(5)}(q)$, $f_1^{(5)}(q)$, $2\mu^{(6)}(q)$ and $\lambda^{(6)}(q)$ are all of parity type $(\frac{3}{4},\frac{1}{4})$.
\end{conj}
Here we consider $2\mu^{(6)}(q)$ instead of $\mu^{(6)}(q)$ because that $c(\mu^{(6)};n)$ is sometimes a half integer instead of an integer (see \eqref{6-mu-half}).

Now we illustrate some examples for Conjecture \ref{conj-mock}. In Theorem \ref{thm-mock-3-f-phi} we prove that $c(f^{(3)};n)\equiv p(n)$ (mod 4) and hence \eqref{pn-half} suggests that $f^{(3)}(q)$ is likely to be of parity type $(\frac{1}{2},\frac{1}{2})$.

As for $g$ being any of the mock theta functions  $f_0^{(5)}(q)$, $f_1^{(5)}(q)$, $2\mu^{(6)}(q)$ and $\lambda^{(6)}(q)$, we  show that there exists $r=0$ or 1 such that  $c(g;2n+r)$ is almost always even, and numerical evidence reveals that the sequence $c(g;2n+1-r)$ takes odd values half of the time. For instance, we show in Theorem \ref{thm-mock-6-lambda} that $c(\lambda^{(6)};2n)$ is odd if and only if $n=k(k+1)/2$ for some integer $k$. Moreover, we find that (see \eqref{6-lambda-psi})
\begin{align}\label{intro-6-lambda-psi}
c(\lambda^{(6)};2n-1) \equiv c(\psi^{(6)};n) \pmod{2}.
\end{align}
 Therefore, the conjecture that $\psi^{(6)}(q)$ is of parity type $(\frac{1}{2},\frac{1}{2})$ is equivalent to the conjecture that $\lambda^{(6)}(q)$ is of parity type $(\frac{3}{4},\frac{1}{4})$.

The paper is organized as follows. In Section \ref{sec-pre} we introduce some notations and collect some necessary results. In particular, we derive formulas for the number of inequivalent elements with fixed norms in the quadratic integer rings $\mathbb{Z}[\sqrt{d}]$ with $d\in \{2,3,6,15\}$. These formulas are fundamental in giving characterizations for odd values of coefficients of mock theta functions of parity type $(1,0)$. Then in Sections \ref{sec-mock-2}--\ref{sec-mock-10} we study parity of coefficients of mock theta functions of orders 2, 3, 5, 6, 7, 8 and 10, respectively. Theorem \ref{thm-main} will follow from the corresponding theorems for the mock theta functions in list. In Section \ref{sec-concluding} we point out  that to prove Conjecture \ref{conj-mock}, we only need to verify it for 15 functions instead of all of them. We also give numerical evidence that supports this conjecture.

\section{Preliminaries}\label{sec-pre}
We recall the Jacobi's triple product identity:
\begin{align*}
j(x;q):=(x)_\infty (q/x)_\infty (q)_\infty =\sum_{n=-\infty}^{\infty} (-1)^nq^{\binom{n}{2}}x^n.
\end{align*}
As some special cases, we let $a$ and $m$ be rational numbers with $m$ positive and define
\begin{align*}
J_{a,m}:=j(q^a;q^m), \quad \overline{J}_{a,m}:=j(-q^a;q^m) \,\, \textrm{and} \,\, J_{m}:=J_{m,3m}=(q^m;q^m)_\infty.
\end{align*}
The following identities will be used without mention (see \cite[Section 2]{Hickerson-Mortenson}):
\begin{align*}
&\overline{J}_{0,1}=2\overline{J}_{1,4}=2\frac{J_{2}^2}{J_{1}}, \quad \overline{J}_{1,2}=\frac{J_{2}^5}{J_{1}^2J_{4}^2}, \quad {J}_{1,2}=\frac{J_{1}^2}{J_{2}}, \quad \overline{J}_{1,3}=\frac{J_2J_{3}^2}{J_1J_6}, \\
& {J}_{1,4}=\frac{J_1J_4}{J_2}, \quad J_{1,6}=\frac{J_1J_6^2}{J_2J_3}, \quad \overline{J}_{1,6}=\frac{J_2^2J_3J_{12}}{J_1J_4J_6}.
\end{align*}
Following Hickerson and Mortenson \cite{Hickerson-Mortenson}, we define
\begin{align}\label{m-defn}
m(x,q,z):=\frac{1}{j(z;q)}\sum_{r=-\infty}^{\infty} \frac{(-1)^rq^{\binom{r}{2}}z^r}{1-q^{r-1}xz},
\end{align}
where $x,z\in \mathbb{C}^{*}:=\mathbb{C} \backslash \{0\}$ with neither $z$ nor $xz$ an integral power of $q$.
\begin{lemma}\label{lem-m-add}
(Cf.\ \cite[Corollary 3.7]{Hickerson-Mortenson}.) For generic $x,z\in \mathbb{C}^{*}$
\begin{align}
m(x,q,z)=m(-qx^2,q^4,z^4)-\frac{x}{q}m(-\frac{x^2}{q},q^4,z^4)-\frac{J_2J_4j(-xz^2;q)j(-xz^3;q)}{xj(xz;q)j(z^4;q^4)j(-qx^2z^4;q^2)}.
\end{align}
\end{lemma}
We also use the notations
\begin{align*}
[z;q]_\infty&=(z;q)_\infty(q/z;q)_\infty, \\
[z_1,z_2,\dots,z_n;q]_\infty&=[z_1;q]_\infty[z_2;q]_\infty\cdots [z_n;q]_\infty.
\end{align*}

Let $d$ be a squarefree integer. Let $K$ be the quadratic field $\mathbb{Q}(\sqrt{d})$, and $\mathcal{O}_K$ be the ring of integers in $K$. We define
\begin{align}
D=\left\{\begin{array}{ll}
d & \text{if $d\equiv 1$ (mod 4)}; \\
4d & \text{if $d\equiv 2,3$ (mod 4)}.
\end{array}\right.
\end{align}
Following \cite[Chapter 3, Eq.\ (8.5)]{BS-book} we define the character of the field $K$ as follows.
\begin{defn}
Given a quadratic field $K=\mathbb{Q}(\sqrt{d})$. Let $\chi_K: \mathbb{Z}\rightarrow \mathbb{Z}$ be defined as follows:\\
(1) If $(x,D)>1$, then $\chi_K(x)=0$; \\
(2) If $(x,D)=1$, then
\begin{align}
\chi_K(x)=\left\{\begin{array}{ll}
\left(\frac{x}{|d|}\right) & \text{if $d\equiv 1$ (mod 4)}; \\
(-1)^{(x-1)/2}\left(\frac{x}{|d|}\right) & \text{if $d\equiv 3$ (mod 4)}; \\
(-1)^{(x^2-1)/8+(x-1)(d'-1)/4}\left(\frac{x}{|d'|}\right) & \text{if $d=2d'$}.
\end{array}\right.
\end{align}
Here $\left(\frac{a}{b}\right)$ is the Jacobi symbol.
\end{defn}
It is not difficult to see that $\chi_K$ is completely multiplicative \cite[p.\ 237]{BS-book}.
\begin{lemma}\label{lem-ideal}
(Cf.\ \cite[p.\ 249, Exercise 17]{BS-book}.) Let $T_{\mathcal{O}_K}(n)$ be the number of ideals in $\mathcal{O}_K$ of norm $n$. Then for $n\geq 1$,
\begin{align}
T_{\mathcal{O}_K}(n)=\sum_{m|n}\chi_K(m).
\end{align}
\end{lemma}
Since $\chi$ is multiplicative, it is clear that $T_{\mathcal{O}_K}(n)$ is also multiplicative, i.e., for $(m,n)=1$,
\begin{align}\label{T-mu}
T_{\mathcal{O}_K}(mn)=T_{\mathcal{O}_K}(m)T_{\mathcal{O}_K}(n).
\end{align}

For our purposes, from now on we will let $d\in \{2,3,6,15\}$. In these cases, we have $\mathcal{O}_K=\mathbb{Z}[\sqrt{d}]$. Recall that the norm of an element $x+y\sqrt{d} \in \mathbb{Q}[\sqrt{d}]$  is given by
\begin{align}
N(x+y\sqrt{d})=x^2-dy^2.
\end{align}
We say that two elements $a,b \in \mathbb{Z}[\sqrt{d}]$ are equivalent if $ab^{-1}$ is a unit in $\mathbb{Z}[\sqrt{d}]$. Let $H_{\mathbb{Z}[\sqrt{d}]}(n)$ be the number of equivalence classes of elements in $\mathbb{Z}[\sqrt{d}]$ with norm $n$. We are going to derive some formulas for $H_{\mathbb{Z}[\sqrt{d}]}(n)$. These formulas will be employed in studying the parity of mock theta functions of parity type $(1,0)$.

Finding elements with norm $m$ in the ring $\mathbb{Z}[\sqrt{d}]$ is equivalent to finding integer solutions of  Pell's equation
\begin{align}\label{Pell-eq}
u^2-dv^2=m.
\end{align}
In general, two integer solutions $(u_1,v_1)$ and $(u_2,v_2)$ of \eqref{Pell-eq} are called equivalent if there exists $(x,y)\in \mathbb{Z}^2$ such that
\begin{align}
x^2-dy^2=1
\end{align}
and
\begin{align}
u_1+v_1\sqrt{d}=(x+y\sqrt{d})(u_2+v_2\sqrt{d}).
\end{align}
In other words,  $(u_1,v_1)$ and $(u_2,v_2)$ are equivalent if and only if $u_1+v_1\sqrt{d}$ and $u_2+v_2\sqrt{d}$ are equivalent in $\mathbb{Z}[\sqrt{d}]$.

We recall the following result found by Andrews, Dyson and Hickerson \cite{ADH}.
\begin{lemma}\label{lem-Pell}
(Cf.\ \cite[Lemma 3]{ADH}.) Let $(x_1,y_1)$ be the fundamental solution of $x^2-dy^2=1$; i.e.\ the solution in which $x_1$ and $y_1$ are minimal positive. If $m>0$, then each equivalence class of solutions of $u^2-dv^2=m$ contains a unique $(u,v)$ with $u>0$ and
\begin{align*}
-\frac{y_1}{x_1+1}u<v\leq \frac{y_1}{x_1+1}u.
\end{align*}
If $m<0$, the corresponding conditions are $v>0$ and
\begin{align*}
-\frac{dy_1}{x_1+1}v<u\leq \frac{dy_1}{x_1+1}{v}.
\end{align*}
\end{lemma}

The following lemma seems to be a basic fact in algebraic number theory. Since we cannot find a reference, we provide a proof here.
\begin{lemma}\label{lem-norm-1}
Let $d$ be a square free integer. If $d$ has a prime divisor $p\equiv 3$ \text{\rm{(mod 4)}}, then $\mathbb{Q}(\sqrt{d})$ does not contain an element of norm $-1$.
\end{lemma}
\begin{proof}
Let $d=pd_1$. Suppose we have some element $x+y\sqrt{pd_1}\in \mathbb{Q}(\sqrt{d})$ with $N(x+y\sqrt{pd_1})=-1$. Let $x=a/b, y=r/s$ where $b,s$ are positive integers and $(a,b)=(r,s)=1$. We have
\begin{align}
N(x+y\sqrt{pd_1})=\frac{a^2}{b^2}-pd_1\frac{r^2}{s^2}=-1.
\end{align}
This implies
\begin{align}
(a^2+b^2)s^2=pd_1b^2r^2. \label{abrs}
\end{align}
Since $\left(\frac{-1}{p}\right)=-1$ and $(a,b)=1$, we have $p\nmid (a^2+b^2)$. Therefore, the $p$-adic order of the left side of \eqref{abrs} must be even, while the $p$-adic order of the right side is odd. This is a contradiction. Thus $\mathbb{Q}(\sqrt{d})$ does not contain an element of norm $-1$.
\end{proof}

Now we present explicit formulas for $H_{\mathbb{Z}[\sqrt{d}]}(n)$ for $d\in \{2,3,6\}$. In each of these cases, $\mathbb{Z}[\sqrt{d}]$ is a principal ideal domain. We first study the relations between $H_{\mathbb{Z}[\sqrt{d}]}(n)$  and $T_{\mathbb{Z}[\sqrt{d}]}(n)$.

Since $\mathbb{Z}[\sqrt{2}]$ contains an element $1+\sqrt{2}$ of norm $-1$, we get
\begin{align}\label{H-Z2}
H_{\mathbb{Z}[\sqrt{2}]}(n)=H_{\mathbb{Z}[\sqrt{2}]}(-n)=T_{\mathbb{Z}[\sqrt{2}]}(n).
\end{align}
For $d\in \{3,6\}$, by Lemma \ref{lem-norm-1} we know that $\mathbb{Q}(\sqrt{d})$ does not contain elements of norm $-1$. Therefore, for any $n\geq 1$, we have $H_{\mathbb{Z}[\sqrt{d}]}(n)H_{\mathbb{Z}[\sqrt{d}]}(-n)=0$ and
\begin{align}\label{H-Z36}
H_{\mathbb{Z}[\sqrt{d}]}(n)+H_{\mathbb{Z}[\sqrt{d}]}(-n)=T_{\mathbb{Z}[\sqrt{d}]}(n), \quad d \in \{3,6\}.
\end{align}

\begin{lemma}\label{lem-Z2}
Let $|n|$ have the prime factorization $|n|=2^ap_1^{e_1}\cdots p_j^{e_j}q_1^{f_1}\cdots q_k^{f_k}$, where $p_i\equiv \pm 1$ \text{\rm{(mod 8)}} and $q_i \equiv \pm 3$ \text{\rm{(mod 8)}}. Then
\begin{align}
H_{\mathbb{Z}[\sqrt{2}]}(n)=\left\{\begin{array}{ll}
0 & \text{if some $f_i$ is odd}, \\
(e_1+1)\cdots (e_j+1) & \text{otherwise}.
\end{array}\right.
\end{align}
\end{lemma}

\begin{lemma}\label{lem-Z3}
Let $|n|$ have the prime factorization $|n|=2^a3^bp_1^{e_1}\cdots p_j^{e_j} q_1^{f_1}\cdots q_k^{f_k}r_1^{g_1}\cdots r_{\ell}^{g_\ell}$, where $p_i \equiv 1$ \text{\rm{(mod 12)}}, $q_i \equiv \pm 5$ \text{\rm{(mod 12)}}, and $r_i \equiv 11$ \text{\rm{(mod 12)}}. \\
$(1)$ For $n>0$ we have
\begin{align}
&H_{\mathbb{Z}[\sqrt{3}]}(n)  \\
=&\left\{\begin{array}{ll}
0 & \text{if some $f_i$ is odd or $a+b+\sum g_i$ odd}, \\
(e_1+1)\cdots (e_j+1)(g_1+1)\cdots (g_\ell+1) & \text{otherwise}.
\end{array}\right. \nonumber
\end{align}
$(2)$ For $n<0$ we have
\begin{align}
&H_{\mathbb{Z}[\sqrt{3}]}(n) \\
=& \left\{\begin{array}{ll}
0 & \text{if some $f_i$ is odd or $a+b+\sum g_i$ even}, \\
(e_1+1)\cdots (e_j+1)(g_1+1)\cdots (g_\ell+1) & \text{otherwise}.
\end{array}\right. \nonumber
\end{align}
\end{lemma}

\begin{lemma}\label{lem-Z6}
Let $|n|$ have the prime factorization $|n|=2^a3^bp_1^{e_1}\cdots p_j^{e_j}q_1^{f_1}\cdots q_k^{f_k}r_1^{g_1}\cdots r_{\ell}^{g_\ell}$, where  $p_i\equiv \pm 7, \pm 11$ \text{\rm{(mod 24)}}, $q_i\equiv 1, 19$ \text{\rm{(mod 24)}}, and $r_i \equiv 5, 23$ \text{\rm{(mod 24)}}. \\
$(1)$ For $n>0$ we have
\begin{align}
&H_{\mathbb{Z}[\sqrt{6}]}(n) \\
=&\left\{\begin{array}{ll}
0 & \text{if some $e_i$ is odd, or $a+\sum g_i$ is odd}, \\
(f_1+1)\cdots (f_k+1)(g_1+1)\cdots (g_\ell+1) & \text{otherwise}.
\end{array}\right. \nonumber
\end{align}
$(2)$ For $n<0$ we have
\begin{align}
&H_{\mathbb{Z}[\sqrt{6}]}(n)  \\
=&\left\{\begin{array}{ll}
0 & \text{if some $e_i$ is odd, or $a+\sum g_i$ is even}, \\
(f_1+1)\cdots (f_k+1)(g_1+1)\cdots (g_\ell+1) & \text{otherwise}.
\end{array}\right. \nonumber
\end{align}
\end{lemma}
The proofs of Lemmas \ref{lem-Z2}--\ref{lem-Z6} can be found in the work of Lovejoy \cite{Lovejoy},  where he linked certain $q$-series related with overpartitions to $\mathbb{Z}[\sqrt{d}]$ with $d\in \{2,3,6\}$. Note that Lovejoy stated almost all these results in the proofs of his Theorems 1-3 in \cite{Lovejoy}, except that we add the cases $n>0$ for $H_{\mathbb{Z}[\sqrt{3}]}(n)$ and $n<0$ for $H_{\mathbb{Z}[\sqrt{6}]}(n)$.

It is possible to prove Lemmas \ref{lem-Z2}--\ref{lem-Z6} in a way slightly different from \cite{Lovejoy}. The key ingredient in our proof is Lemma \ref{lem-ideal}. To save space, here we only present a proof for Lemma \ref{lem-Z6}. Lemmas \ref{lem-Z2} and \ref{lem-Z3} can be proved in an analogous way.
\begin{proof}[Proof of Lemma \ref{lem-Z6}]
Note that for $(x,24)=1$,
\begin{align}
\chi_{\mathbb{Q}(\sqrt{6})}(x)=(-1)^{(x^2-1)/6+(x-1)/2}\left(\frac{x}{3}\right)=\left\{\begin{array}{ll}
1 & \text{if $x\equiv 1,5, 19, 23$ \rm{(mod 24)}}, \\
-1 & \text{if $x\equiv 7,11, 13, 17$ \rm{(mod 24)}}.
\end{array}\right. \nonumber
\end{align}
Therefore, for prime $p>3$ we deduce that
\begin{align}\label{Z3-T}
T_{\mathbb{Z}[\sqrt{6}]}(p^e)=\sum_{i=0}^e \chi_{\mathbb{Q}(\sqrt{6})}(p^i)=\left\{\begin{array}{ll}
0 &\text{if $e$ is odd and $p\equiv \pm 7,\pm 11$ \rm{(mod 24)}}, \\
1 & \text{if $e$ is even and $p\equiv \pm 7,\pm 11$ \rm{(mod 24)}}, \\
e+1 &\text{ if $p\equiv \pm 1, \pm 5$ \rm{(mod 24)}}.
\end{array}\right.
\end{align}
Note that $T_{\mathbb{Z}[\sqrt{6}]}(2^e)=T_{\mathbb{Z}[\sqrt{6}]}(3^e)=1$. If $a$ is odd, then $2^a\equiv 2$ (mod 3) and $x^2-6y^2=2^a$ has no integer solutions. Thus $H_{\mathbb{Z}[\sqrt{6}]}(2^a)=0$ for odd $a$. Similarly, if $a$ is even, then $H_{\mathbb{Z}[\sqrt{6}]}(-2^a)=0$. By \eqref{H-Z36} we get $H_{\mathbb{Z}[\sqrt{6}]}((-2)^a)=1$. In the same way, we can prove that $H_{\mathbb{Z}[\sqrt{6}]}(3^b)=1$, $H_{\mathbb{Z}[\sqrt{6}]}(p^e)=1$ for prime $p\equiv \pm 7, \pm 11$ (mod 24) and $e$ even, $H_{\mathbb{Z}[\sqrt{6}]}(p^e)=0$ for prime $p\equiv \pm 7, \pm 11$ (mod 24) and $e$ odd, $H_{\mathbb{Z}[\sqrt{6}]}(q^f)=f+1$ for prime $q\equiv 1,19$ (mod 24), and $H_{\mathbb{Z}[\sqrt{6}]}((-r)^g)=g+1$ for prime $r\equiv 5,23$ (mod 24).

Therefore, by unique factorization, a number $n\neq 0$ can be the norm of an element in $\mathbb{Z}[\sqrt{6}]$ if and only if $n=(-2)^a3^bp_1^{e_1}\cdots p_j^{e_j} q_1^{f_1}\cdots q_k^{f_k}(-r_1)^{g_1}\cdots (-r_{\ell})^{g_\ell}$, where $p_i \equiv \pm 7, \pm 11$ \text{\rm{(mod 24)}}, $q_i \equiv 1, 19$ \text{\rm{(mod 24)}}, and $r_i \equiv 5, 23$ \text{\rm{(mod 24)}}. Tracking the sign of $n$, we get the desired conclusion.
\end{proof}

As for $d=15$, since $\mathbb{Z}[\sqrt{15}]$ is not a unique factorization domain, finding an explicit formula for $H_{\mathbb{Z}[\sqrt{15}]}(n)$ similar to Lemmas \ref{lem-Z2}--\ref{lem-Z6} is more difficult. Nevertheless, for $n$ in some special arithmetic progressions, we can find explicit formulas for $H_{\mathbb{Z}[\sqrt{15}]}(n)$.
\begin{lemma}\label{lem-Tn-15}
Let $n$ have the prime factorization $n=2^a3^b5^cp_1^{e_1}\cdots p_j^{e_j}q_1^{f_1}\cdots q_k^{f_k}$ where the primes $p_i\equiv \pm 1, \pm 7, \pm 11, \pm 17$ \text{\rm{(mod 60)}} and $q_i\equiv \pm 13, \pm 19, \pm 23,\pm 29$ \text{\rm{(mod 60)}}. Then
\begin{align}\label{Tn-15}
T_{\mathbb{Z}[\sqrt{15}]}(n)=\left\{\begin{array}{ll}
0 &\text{if some $f_i$ odd}, \\
(e_1+1)\cdots (e_k+1) & \text{otherwise}.
\end{array}\right.
\end{align}
\end{lemma}
\begin{proof}
By definition we have
\begin{align}
\chi_{\mathbb{Q}(\sqrt{15})}(x)=\left\{\begin{array}{ll}
0 &\text{if $(x,60)>1$}, \\
(-1)^{(x-1)/2}\left(\frac{x}{15}\right) & \text{if $(x,60)=1$}.
\end{array}\right.
\end{align}
Hence for prime $p>5$ we have
\begin{align}
\chi_{\mathbb{Q}(\sqrt{15})}(p)=\left\{\begin{array}{ll}
1 &\text{if $p\equiv \pm 1, \pm 7, \pm 11, \pm 17$ \rm{(mod 60)}}, \\
-1 & \text{if $p\equiv \pm 13, \pm 19, \pm 23, \pm 29$ \rm{(mod 60)}}.
\end{array}\right.
\end{align}
This implies
\begin{align}
T_{\mathbb{Z}[\sqrt{15}]}(p^e)&=\sum_{i=0}^e \chi_{\mathbb{Q}(\sqrt{15})}(p^i) \nonumber \\
&=\left\{\begin{array}{ll}
0 &\text{if $e$ is odd and $p\equiv \pm 13, \pm 19, \pm 23, \pm 29$ \rm{(mod 60)}},\\
1 & \text{if $e$ is even and $p\equiv \pm 13, \pm 19, \pm 23, \pm 29$ \rm{(mod 60)}}, \\
e+1 &\text{if $p\equiv \pm 1, \pm 7, \pm 11, \pm 17$ \rm{(mod 60)}}.
\end{array}\right.
\end{align}
It is obvious that $T_{\mathbb{Z}[\sqrt{15}]}(2^e)=T_{\mathbb{Z}[\sqrt{15}]}(3^e)=T_{\mathbb{Z}[\sqrt{15}]}(5^e)=1$. The assertion then follows from the multiplicativity of $T_{\mathbb{Z}[\sqrt{15}]}(n)$.
\end{proof}

\begin{lemma}\label{lem-PID}
Let $I$ be an ideal of $\mathbb{Z}[\sqrt{15}]$. If $N(I)\equiv \pm 1$ \text{\rm{(mod 5)}}, then $I$ must be a principal ideal. If $N(I)\equiv \pm 2$ \text{\rm{(mod 5)}}, then $I$ is not a principal ideal.
\end{lemma}
\begin{proof}
If $I=(x+y\sqrt{15})$ is a principal ideal, then we have $N(I)=|N(x+y\sqrt{15})|=|x^2-15y^2|$. This implies that $N(I)\equiv 0, 1, 4$ (mod 5). Now it suffices to show that if $N(I)\equiv \pm 1$ (mod 5), then $I$ must be a principal ideal. We prove this by contradiction.

Suppose that $I$ is not a principal ideal. Note that the ideal $(3,\sqrt{15})$ is not principal. Since the class number of $\mathbb{Q}(\sqrt{15})$ is 2, there exist some nonzero elements $\alpha,\beta\in \mathbb{Z}[\sqrt{15}]$ such that $(\alpha)I=(\beta)(3,\sqrt{15})$. Taking ideal norms on both sides, we get
\begin{align}
|N(\alpha)|N(I)=3|N(\beta)|.
\end{align}
Let $\alpha=x+y\sqrt{15}$ and $\beta=a+b\sqrt{15}$. Then we get
\begin{align}\label{inf-1}
|x^2-15y^2|N(I)=3|a^2-15b^2|.
\end{align}
This implies $x^2\equiv \pm 2 a^2$ (mod 5). Since $\left(\frac{2}{5}\right)=\left(\frac{-2}{5}\right)=-1$, we must have $5|x$ and $5|a$. Let $x=5x_1$ and $a=5a_1$. We get
\begin{align}
|5x_1^2-3y^2|N(I)=3|5a_1^2-3b^2|.
\end{align}
This implies $y^2\equiv \pm 2 b^2$ (mod 5). Again we deduce that $5|y$ and $5|b$. Let $y=5y_1$ and $b=5b_1$. We get
\begin{align}\label{inf-2}
|x_1^2-3y_1^2|N(I)=3|a_1^2-15b_1^2|.
\end{align}
This has the same form with \eqref{inf-1}. We can continue the above process infinitely many times. This leads to a contradiction. Therefore, when $N(I)\equiv \pm 1$ (mod 5), $I$ must be principal.
\end{proof}

\begin{lemma}\label{lem-15-HT}
For $n\geq 0$ we have
\begin{align}
H_{\mathbb{Z}[\sqrt{15}]}(20n+r) &=T_{\mathbb{Z}[\sqrt{15}]}(20n+r), \quad r\in \{1,9\}, \label{15-HT-1}\\
H_{\mathbb{Z}[\sqrt{15}]}(40n+r) &=T_{\mathbb{Z}[\sqrt{15}]}(40n+r), \quad r\in \{26,34\}. \label{15-HT-2}
\end{align}
\end{lemma}
\begin{proof}
We have already seen from Lemma \ref{lem-PID} that if an ideal $I$ of $\mathbb{Z}[\sqrt{15}]$ has norm $N(I)\equiv \pm 1$ (mod 5), then $I=(\alpha)$ for some $\alpha\in \mathbb{Z}[\sqrt{15}]$. We then have $N(\alpha)=\pm N(I)$. Note that the residue of $15y^2-x^2$ modulo 4 cannot be 1, and its residue modulo 8 cannot be 2. Therefore, if $I=(\alpha)$ and $N(I)\equiv 1$ (mod 4) or $N(I)\equiv 2$ (mod 8), then we must have $N(\alpha)=N(I)$. These facts imply that the number of inequivalent elements with norm congruent to 1 or 9 modulo 20, or congruent to 26 or 34 modulo 40, is equal to the number of ideals with the same norm. This gives \eqref{15-HT-1} and \eqref{15-HT-2}.
\end{proof}

Finally, to give estimates like \eqref{intro-mock-2-A-estimate} for mock theta functions of parity type $(1,0)$, we establish the following result. This generalizes and improves Chen's estimate in \cite[p.\ 1083]{Chen}, which he used for deriving \eqref{intro-Chen-G} and \eqref{intro-Chen-H}.
\begin{lemma}\label{lem-estimate}
Let $A,B$ be positive integers with $A>B$ and $(A,B)=1$.  Let
\begin{align*}
\gamma(N)=\sum_{\begin{smallmatrix}
0\leq n \leq N \\ An+B=m^2p^{4a+1}\\ p \text{ is prime and $p\nmid m$}
\end{smallmatrix}} 1.
\end{align*}
Then we have
\begin{align}\label{gamma-result}
\gamma(N)=\frac{\pi^2}{6}\prod\limits_{p|A}(1+p^{-1})\frac{N}{\log N}+O\left(\frac{N}{\log^2N}  \right).
\end{align}
\end{lemma}
\begin{proof}
Here we follow and  refine the arguments of Chen \cite{Chen}.
Let $\pi(x)$ be the number of primes no greater than $x$ and we denote $M=AN+B$ for convenience. Now we split the above sum into two parts according to $a=0$ and $a\geq 1$, respectively.

The sum over $a\geq 1$ is bounded by
\begin{align*}
\ll\sum_{1\leq a  \ll \log N} \sum_{m\leq \sqrt{M}} \pi \left( \left(\frac{M}{m^2}\right)^{\frac{1}{4a+1}}\right) \ll \log N \sum_{m\leq \sqrt{M}} \left(\frac{N}{m^2} \right)^{\frac{1}{5}} \ll N^{\frac{1}{2}}\log N,
\end{align*}
which is negligible.

The sum over $a=0$ is
\begin{align*}
\sum_{\begin{smallmatrix} pm^2\leq M, pm^2\equiv B \, \text{\rm{(mod $A$)}},  \\p \text{ is prime and $p\nmid m$} \end{smallmatrix}} 1=\sum_{\begin{smallmatrix}
m\leq \sqrt{M} \\ (m,A)=1
\end{smallmatrix}} \sum_{\begin{smallmatrix}
p\leq M/m^2 \\ p\equiv B/m^2 \, \text{\rm{(mod $A$)}}
\end{smallmatrix}} 1
-\sum_{\begin{smallmatrix}
m\leq \sqrt{M} \\ (m,A)=1
\end{smallmatrix}} \sum_{\begin{smallmatrix}
p\leq M/m^2, p|m \\  p\equiv B/m^2 \, \text{\rm{(mod $A$)}}
\end{smallmatrix}} 1.
\end{align*}
Let $\omega(n)$ denote the number of distinct prime factors of $n$. Note that
\begin{align*}
\sum_{\begin{smallmatrix}
m\leq \sqrt{M} \\ (m,A)=1
\end{smallmatrix}} \sum_{\begin{smallmatrix}
p\leq M/m^2, p|m \\  p\equiv B/m^2 \, \text{\rm{(mod $A$)}}
\end{smallmatrix}} 1 \leq \sum_{\begin{smallmatrix}
m\leq \sqrt{M} \\ (m,A)=1
\end{smallmatrix}} \sum_{p|m} 1 \leq \sum_{\begin{smallmatrix}
m\leq \sqrt{M}
\end{smallmatrix}} \omega(m) \ll N^{\frac{1}{2}}\log\log N,
\end{align*}
where in the last inequality we used the fact that \cite[Eq.\ (22.10.3)]{Hardy-Wright-book}
\begin{align}
\sum_{n\leq x}\omega(n)=x\log\log x+O(x).
\end{align}
Next, we focus on the term
\begin{align}\label{main-term}
\sum_{\begin{smallmatrix}
m\leq \sqrt{M} \\ (m,A)=1
\end{smallmatrix}} \sum_{\begin{smallmatrix}
p\leq M/m^2 \\ p\equiv B/m^2 \, \text{\rm{(mod $A$)}}
\end{smallmatrix}} 1.
\end{align}
The contribution for $\log^2 M<m$ is absorbed into the error term since
\begin{align*}
\sum_{\begin{smallmatrix}
\log^2 M<m\leq \sqrt{M} \\ (m,2)=1
\end{smallmatrix}} \frac{M}{m^2} \ll \frac{M}{\log^2 M} \ll \frac{N}{\log^2N}.
\end{align*}
Let $\phi(k)$ be Euler's totient function. The sum over $m\leq \log^2M$ is
\begin{align}
&\sum_{\begin{smallmatrix}
m\leq \log^2 M \\ (m,A)=1
\end{smallmatrix}} \frac{M}{\phi(A)m^2\log(\frac{M}{m^2})}\Bigg(1+O\Big(\frac{1}{\log(\frac{M}{m^2})} \Big) \Bigg) \label{main-estimate} \\
=&\frac{1}{\phi(A)}\sum_{\begin{smallmatrix}
m\leq \log^2M \\ (m,A)=1
\end{smallmatrix}}\frac{M}{m^2\log M}\left(1+O\left(\frac{\log m}{\log M} \right) \right) \nonumber \\
=&\frac{1}{\phi(A)} \sum_{(m,A)=1}\frac{M}{m^2\log M}+O\Big(\sum_{m>\log^2 M}\frac{M}{m^2\log M} \Big)+O\left(\frac{M}{\log^2M}  \right)\\
=&\frac{1}{\phi(A)} \sum_{(m,A)=1}\frac{M}{m^2\log M}+O\Big(\frac{M}{\log^2M}\Big). \nonumber
\end{align}
Here for the third term of the last second line we used the fact that $\sum_{m=1}^\infty \frac{\log m}{m^2}$ converges.

Thus
\begin{align*}
\gamma(N)&=\frac{1}{\phi(A)}\sum_{\begin{smallmatrix}
m=1 \\ (m,A)=1
\end{smallmatrix}}^\infty \frac{AN+B}{m^2\log (AN+B)} +O\left(\frac{N}{\log^2N} \right) \\
&=\Big(\frac{A}{\phi(A)}\sum_{\begin{smallmatrix}
m=1 \\ (m,A)=1
\end{smallmatrix}}^\infty \frac{1}{m^2}  \Big)\frac{N}{\log N}+O\left( \frac{N}{\log^2N} \right).
\end{align*}
Note that
\begin{align*}
\sum_{\begin{smallmatrix}
m=1 \\ (m,A)=1
\end{smallmatrix}}^\infty \frac{1}{m^2}=\prod\limits_{p\nmid A}(1-p^{-2})^{-1}=\prod\limits_{p|A}(1-p^{-2})\zeta(2),
\end{align*}
where $\zeta(s)=\sum_{n=1}^\infty \frac{1}{n^s}$ ($\mathrm{Re} (s)>1$) is the Riemann zeta function. Next, using the fact that $\zeta(2)=\frac{\pi^2}{6}$ and $\phi(A)=A\prod\limits_{p|A}(1-p^{-1})$, we get  \eqref{gamma-result}.
\end{proof}
\section{Mock theta functions of order 2}\label{sec-mock-2}
There are three classical mock theta functions of order 2:
\begin{align}
&A^{(2)}(q):=\sum_{n=0}^{\infty}\frac{q^{n+1}(-q^2;q^2)_{n}}{(q;q^2)_{n+1}}=\sum_{n=0}^{\infty}\frac{q^{(n+1)^2}(-q;q^2)_{n}}{(q;q^2)_{n+1}^2}, \label{mock-2-A-defn} \\
&B^{(2)}(q):=\sum_{n=0}^{\infty}\frac{q^{n}(-q;q^2)_{n}}{(q;q^2)_{n+1}}=\sum_{n=0}^{\infty}\frac{q^{n^2+n}(-q^2;q^2)_{n}}{(q;q^2)_{n+1}^2}, \label{mock-2-B-defn} \\
&\mu^{(2)}(q):=\sum_{n=0}^{\infty}\frac{(-1)^n(q;q^2)_{n}q^{n^2}}{(-q^2;q^2)_{n}^2}. \label{mock-2-mu-defn}
\end{align}
These functions appeared in  Ramanujan's Lost Notebook \cite{lostnotebook}. They were recognized as mock theta functions for the first time in the work of Gordon and McIntosh \cite[p.\ 120, Eq. (5.1)]{Gordon-McIntosh-Survey} and  McIntosh's paper \cite{McIntosh-2007Canad}. 
\begin{theorem}\label{thm-mock-2-A}
The number $c(A^{(2)};n)$ is odd if and only if $8n-1=p^{4a+1}m^2$, where $p$ is a prime, $m$ is a positive integer and $p\nmid m$. Furthermore,
\begin{align}\label{mock-2-A-estimate}
\#\left\{0\leq n\leq N: c(A^{(2)};n)\equiv 1 \,\, (\mathrm{mod \,\, 2}) \right\}=\frac{\pi^2}{4}\frac{N}{\log N}+O\left(\frac{N}{\log^2N} \right).
\end{align}
\end{theorem}
\begin{proof}
We utilize the following Hecke-type series representation found by Cui, Gu and Hao \cite{CGH}:
\begin{align}
A^{(2)}(q)=q\frac{(q^2;q^2)_\infty}{(q;q)_\infty^2}\sum_{n=0}^\infty (-1)^nq^{2n^2+3n}(1-q^{2n+2})\sum_{j=0}^nq^{-j^2-j}. \label{Hecke-mock-2-A}
\end{align}
By the binomial theorem, we have
\begin{align}
A^{(2)}(q)\equiv \sum_{n=0}^\infty q^{2n^2+3n+1}(1+q^{2n+2})\sum_{j=0}^nq^{-j^2-j} \pmod{2}.
\end{align}
It follows that
\begin{align}
&\sum_{n=0}^\infty c(A^{(2)};n)q^{8n-1} \nonumber \\
\equiv & \sum_{n=0}^\infty \sum_{j=0}^n \left(q^{(4n+3)^2-2(2j+1)^2}+q^{(4n+5)^2-2(2j+1)^2}\right) \nonumber \\
\equiv  & \frac{1}{2} \left(\sum_{n=0}^\infty \sum_{j=-n-1}^n q^{(4n+3)^2-2(2j+1)^2}+\sum_{n=0}^\infty \sum_{j=-n-1}^nq^{(4n+5)^2-2(2j+1)^2} \right) \pmod{2}. \label{mock-2-A-reduce}
\end{align}
Note that the fundamental solution of $x^2-2y^2=1$ is $(x_1,y_1)=(3,2)$. If $m>0$, by Lemma \ref{lem-Pell} we know that each equivalence class of solutions of $u^2-2v^2=m$ contains a unique $(u,v)$ with $u>0$ and
\begin{align}\label{mock-2-A-ineq}
-\frac{1}{2}u<v\leq \frac{1}{2}u.
\end{align}
Let $m=8n-1$. Then such $(u,v)$ must satisfy $u\equiv v\equiv 1$ (mod 2). Moreover, as $n\geq 1$, we have $u\geq 3$. If $u\equiv 1$ (mod 4), we write $u=4n+5$ and $v=2j+1$. Then the condition \eqref{mock-2-A-ineq} is equivalent to
\begin{align*}
-n-1\leq j \leq n.
\end{align*}
Similarly, if $u\equiv 3$ (mod 4), we write $u=4n+3$ and $v=2j+1$. Then the condition \eqref{mock-2-A-ineq} is again equivalent to $-n-1\leq j \leq n$.

Hence by \eqref{mock-2-A-reduce} we conclude that
\begin{align}\label{mock-2-A-equiv}
c(A^{(2)};n)\equiv \frac{1}{2}{H}_{\mathbb{Z}[\sqrt{2}]}(8n-1) \pmod{2}.
\end{align}
Let $8n-1$ have the prime factorization $p_1^{e_1}\cdots p_j^{e_j}q_1^{f_1}\cdots q_k^{f_k}$ where $p_i\equiv \pm 1$ (mod 8) and $q_j\equiv \pm 3$ (mod 8). By Lemma \ref{lem-Z2}, we see that $\frac{1}{2}H_{\mathbb{Z}[\sqrt{2}]}(8n-1)$ is odd if and only if $f_i$ are all even, and exactly one of $e_i$ is congruent to 1 modulo 4 while others are all even. In other words, $c(A^{(2)};n)$ is odd if and only if $8n-1=p^{4a+1}m^2$ where $p\equiv -1$ (mod 8) is a prime and $p\nmid m$.

Let
\begin{align*}
\gamma(N)=\#\left\{0\leq n < N: 8n+7=p^{4a+1}m^2, \text{$p$ is a prime and $p\nmid m$} \right\}.
\end{align*}
Then from the above we have
\begin{align*}
\#\left\{0\leq n\leq N: c(A^{(2)};n)\equiv 1 \,  (\mathrm{mod \,\, 2}) \right\}=\gamma(N).
\end{align*}
Now applying Lemma \ref{lem-estimate} with $(A,B)=(8,7)$ we get \eqref{mock-2-A-estimate}.
\end{proof}

\begin{theorem}\label{thm-mock-2-B}
The coefficient $c(B^{(2)};n)$ is odd if and only if $n=2k^2+2k$ for some $k\geq 0$. Moreover,
\begin{align}
\#\left\{0\leq n\leq N:c(B^{(2)};n)\equiv 1 \, (\mathrm{mod \,\, 2}) \right\}=\left\lfloor \frac{\sqrt{2N+1}+1}{2}\right\rfloor.
\end{align}
\end{theorem}
\begin{proof}
We use the following Hece-type series representation found by Cui, Gu and Hao \cite{CGH}:
\begin{align}
B^{(2)}(q)=\frac{J_2}{J_1^2}\sum_{n=0}^\infty (-1)^nq^{2n^2+2n}\sum_{j=-n}^n q^{-j^2}.
\end{align}
It follows that
\begin{align*}
B^{(2)}(q)\equiv \sum_{n=0}^\infty q^{2n^2+2n} \pmod{2}.
\end{align*}
The theorem then follows immediately.
\end{proof}

Let $p_{-k}(n)$ be defined as
\begin{align}
\sum_{n=0}^\infty p_{-k}(n)q^n=\frac{1}{(q;q)_\infty^k}.
\end{align}
\begin{theorem}
For any $n\geq 0$, we have $c(\mu^{(2)};n)\equiv p_{-3}(n)$ \text{\rm{(mod 4)}}.
\end{theorem}
\begin{proof}
We use the following Appell-Lerch series representation:
\begin{align}
\mu^{(2)}(q)=2\frac{(q;q^2)_\infty}{(q^2;q^2)_\infty}\sum_{n=-\infty}^\infty \frac{q^{2n^2+n}}{1+q^{2n}}. \label{mu-H}
\end{align}
This formula appeared on pages 8 and 29 in \cite{lostnotebook} and can also be found in \cite{McIntosh-2007Canad}.

From \eqref{mu-H} we deduce that
\begin{align*}
\mu^{(2)}(q)=\frac{(q;q^2)_\infty}{(q^2;q^2)_\infty}\left(1+4\sum_{n=1}^\infty \frac{q^{2n^2+n}}{1+q^{2n}}\right).
\end{align*}
Therefore, by the binomial theorem we get
\begin{align*}
\mu^{(2)}(q)\equiv \frac{1}{(q;q)_\infty^3} \pmod{4}.
\end{align*}
As a consequence, we have $c(\mu^{(2)};n)\equiv p_{-3}(n)$ (mod 4).
\end{proof}
From \cite[Theorems 3.5 and 3.6]{BYZ} we get the following consequence.
\begin{corollary}
 For each fixed $c>2\log 2$ and for $N$ large enough,
\begin{align}
\#\left\{0\leq n\leq N: c(\mu^{(2)};n)\equiv 1  \, (\mathrm{mod \,\, 2}) \right\}\geq N^{\frac{1}{2}-\frac{c}{\log\log N}}.
\end{align}
For each fixed $c$ with $c<\frac{1}{\sqrt{2}}$, and for $N$ large enough,
\begin{align}
\#\left\{0\leq n\leq N: c(\mu^{(2)};n)\equiv 0 \, (\mathrm{mod \,\, 2}) \right\}\geq c\sqrt{N}.
\end{align}
\end{corollary}
However, since such estimates are still far from proving Conjecture \ref{conj-mock}, we will not state them in our theorems.

\section{Mock theta functions of order 3}\label{sec-mock-3}
Ramanujan gave seven mock theta functions of order 3 in his last letter to Hardy and lost notebook \cite{lostnotebook}:
\begin{align*}
&f^{(3)}(q):=\sum_{n=0}^{\infty}\frac{q^{n^2}}{(-q;q)_{n}^2}, 
\quad \phi^{(3)}(q):=\sum_{n=0}^{\infty}\frac{q^{n^2}}{(-q^2;q^2)_{n}},   
\quad \psi^{(3)}(q):=\sum_{n=1}^{\infty}\frac{q^{n^2}}{(q;q^2)_{n}},\\ 
&\chi^{(3)}(q):=\sum_{n=0}^{\infty}\frac{q^{n^2}(-q;q)_{n}}{(-q^3;q^3)_{n}},  
\quad \omega^{(3)}(q):=\sum_{n=0}^{\infty}\frac{q^{2n(n+1)}}{(q;q^2)_{n+1}^2}, \\ 
& \nu^{(3)}(q):=\sum_{n=0}^{\infty}\frac{q^{n(n+1)}}{(-q;q^2)_{n+1}}, 
\quad \rho^{(3)}(q):=\sum_{n=0}^{\infty}\frac{q^{2n(n+1)}(q;q^2)_{n+1}}{(q^3;q^6)_{n+1}}. 
\end{align*}
\begin{theorem}\label{thm-mock-3-f-phi}
For $n\geq 0$ we have $c(f^{(3)};n)\equiv p(n)$ \text{\rm{(mod 4)}} and  $c(\phi^{(3)};n)\equiv p(n)$ \text{\rm{(mod 2)}}.
\end{theorem}
\begin{proof}
We use the following Appell-Lerch series representation due to Watson \cite{Watson}:
\begin{align}
f^{(3)}(q)&=\frac{2}{(q;q)_\infty}\sum_{n=-\infty}^\infty \frac{(-1)^nq^{n(3n+1)/2}}{1+q^n}, \label{mock-3-f-A} \\
\phi^{(3)}(q)&=\frac{1}{(q;q)_\infty}\sum_{n=-\infty}^\infty \frac{(-1)^nq^{n(3n+1)/2}(1+q^n)}{1+q^{2n}}. \label{mock-3-phi-A}
\end{align}
It follows that
\begin{align*}
f^{(3)}(q)=\frac{1}{(q;q)_\infty}\left(1+4\sum_{n=1}^\infty \frac{(-1)^nq^{n(3n+1)/2}}{1+q^n}\right) \equiv \frac{1}{(q;q)_\infty} \pmod{4}, \\
\phi^{(3)}(q)=\frac{1}{(q;q)_\infty}\left(1+2\sum_{n\neq 0} \frac{(-1)^nq^{n(3n+1)/2}}{1+q^{2n}}\right)\equiv \frac{1}{(q;q)_\infty} \pmod{2}.
\end{align*}
This implies the desired congruences.
\end{proof}

\begin{rem}
We can then use \eqref{best-even} and \eqref{best-odd} to derive some estimates for the number of even/odd values of $c(f^{(3)};n)$ and $c(\phi^{(3)};n)$.
Since these estimations follow directly from known results for $p(n)$ and are far from proving Conjecture \ref{conj-mock}, we will not state them in our theorems.
\end{rem}

\begin{theorem}\label{thm-mock-3-psi}
The coefficient $c(\psi^{(3)};n)$ is odd if and only if $24n-1=p^{4a+1}m^2$ for some prime $p$ and integer $m$ with $p\nmid m$. We have
\begin{align}\label{mock-3-psi-estimate}
\#\left\{0\leq n\leq N: c(\psi^{(3)};n) \equiv 1 \text{ \rm{(mod  2)}} \right\}=\frac{\pi^2}{3}\frac{N}{\log N}+O\left( \frac{N}{\log^2 N} \right).
\end{align}
\end{theorem}
\begin{proof}
From \cite[Eq.\ (4.3)]{Liu2013IJNT} we find
\begin{align}
\sum_{n=0}^\infty \frac{(-1)^nq^{n^2}}{(q;q^2)_n}=\sum_{n=0}^\infty \sum_{j=-n}^n(-1)^n(1-q^{4n+2})q^{3n^2+n-2j^2-j}. \label{Liu-id}
\end{align}
By definition and \eqref{Liu-id} we deduce that
\begin{align}
\psi^{(3)}(q)&=\sum_{n=0}^\infty \frac{q^{n^2}}{(q;q^2)_n}-1 \nonumber \\
& \equiv \sum_{n=0}^\infty \sum_{j=-n}^n\left(q^{3n^2+n-2j^2-j}+q^{3n^2+5n+2-2j^2-j} \right)-1 \pmod{2}.
\end{align}
This implies
\begin{align}\label{3-psi-key}
& \sum_{n=1}^\infty c(\psi^{(3)};n)q^{48n-2} \equiv \frac{1}{2}\left(\sum_{n=0}^\infty \sum_{j=-n}^{n-1} q^{(12n+2)^2-6(4j+1)^2}+\sum_{n=0}^\infty \sum_{j=-n+1}^nq^{(12n+2)^2-6(4j-1)^2} \right. \nonumber \\
&\quad \left.+\sum_{n=0}^\infty \sum_{j=-n-1}^n q^{(12n+10)^2-6(4j+1)^2}+\sum_{n=0}^\infty \sum_{j=-n}^{n+1} q^{(12n+10)^2-6(4j-1)^2}  \right) \pmod{2}.
\end{align}
The fundamental solution of $x^2-6y^2=1$ is $(x_1,y_1)=(5,2)$. By Lemma \ref{lem-Pell} and \eqref{3-psi-key} and arguing in a way similar to the proof of Theorem \ref{thm-mock-2-A}, we conclude that
\begin{align}\label{3-psi-H}
c(\psi^{(3)};n)\equiv \frac{1}{2}H_{\mathbb{Z}[\sqrt{6}]}(48n-2) \pmod{2}.
\end{align}
Let $48n-2$ have the prime factorization $48n-2=2p_1^{e_1}\cdots p_j^{e_j}q_1^{f_1}\cdots q_k^{f_k}r_1^{g_1}\cdots r_\ell^{g_\ell}$ where the primes $p_i\equiv \pm 7, \pm 11$ (mod 24), $q_j\equiv 1, 19$ (mod 24), and $r_i\equiv 5,23$ (mod 24). By Lemma \ref{lem-Z6}, we know that $\frac{1}{2}H_{\mathbb{Z}[\sqrt{6}]}(48n-2)$ is odd if and only if all of $f_i$ and $g_i$ are even except for one $g_s\equiv 1$ (mod 4). Thus the condition is equivalent to that $24n-1=p^{4a+1}m^2$ for some prime $p$ and integer $m$ with $p\nmid m$.

Applying Lemma \ref{lem-estimate} with $(A,B)=(24,23)$, we get \eqref{mock-3-psi-estimate}.
\end{proof}
We remark here that the first statement of Theorem \label{thm-mock-3-psi} also follows from Theorem 1.3 (ii) and Theorem 1.5 in the work of Andrews, Garvan and Liang \cite{AGL}.

\begin{theorem}\label{thm-mock-3-nu}
We have $c(\nu^{(3)};n)\equiv a_3(n)$ \text{\rm{(mod 2)}} where $a_3(n)$ is the number of $3$-core partitions of $n$. Moreover, $c(\nu^{(3)};n)$ is odd if and only if $3n+1=k^2$ for some integer $k$. We have
\begin{align}
\#\left\{0\leq n\leq N:c(\nu^{(3)};n)\equiv 1 \, (\mathrm{mod \,\, 2}) \right\}=\lfloor \sqrt{3N+1} \rfloor-\left\lfloor \frac{\sqrt{3N+1}}{3}\right\rfloor.
\end{align}
\end{theorem}
\begin{proof}
Recall the following Appell-Lerch series representation of $\nu^{(3)}(q)$ in \cite{Watson}:
\begin{align}
\nu^{(3)}(q)=\frac{1}{(q;q)_\infty}\sum_{n=0}^\infty (-1)^nq^{3n(n+1)/2}\frac{1-q^{2n+1}}{1+q^{2n+1}}.
\end{align}
We deduce that
\begin{align}\label{3-nu-reduction}
\nu^{(3)}(q) \equiv \frac{1}{(q;q)_\infty}\sum_{n=0}^\infty q^{3n(n+1)/2}=\frac{1}{(q;q)_\infty}\cdot \frac{(q^6;q^6)_\infty^2}{(q^3;q^3)_\infty}\equiv \frac{(q^3;q^3)_\infty^3}{(q;q)_\infty} \pmod{2}.
\end{align}
From \cite[Eq.\ (2.1)]{t-core} we know that
\begin{align}\label{3-core-gen}
\sum_{n=0}^\infty a_3(n)q^n=\frac{(q^3;q^3)_\infty^3}{(q;q)_\infty}.
\end{align}
Therefore, \eqref{3-nu-reduction} implies $c(\nu^{(3)};n)\equiv a_3(n)$ (mod 2).

Next, we recall an explicit formula for $a_3(n)$. Let $3n+1$ have the prime factorization $3n+1=p_1^{e_1}\cdots p_j^{e_j}q_1^{f_1}\cdots q_k^{f_k}$ where $p_i\equiv 1$ (mod 3), and $q_i\equiv 2$ (mod 3).  Hirschhorn and Sellers \cite{HS} found that
\begin{align}\label{3-core-formula}
a_3(n)=\left\{\begin{array}{ll}
(e_1+1)\cdots (e_j+1) & \text{if all $f_i$ are even}; \\
0 & \text{otherwise}.
\end{array}\right.
\end{align}
Therefore, $a_3(n)$ is odd if and only if all of $e_i$ and $f_i$ are even, i.e., $3n+1=k^2$. This proves the first half of the theorem. The second half of the theorem follows immediately.
\end{proof}

\begin{theorem}\label{thm-mock-3-omega}
We have $c(\omega^{(3)};2n+1)\equiv 0$ \text{\rm{(mod 2)}} and $c(\omega^{(3)};2n)\equiv a_3(n)$ \text{\rm{(mod 2)}}. Moreover, $c(\omega^{(3)};2n)$ is odd if and only if $3n+1=k^2$ for some integer $k$. We have
\begin{align}
\#\left\{0\leq n\leq N:c(\omega^{(3)};2n)\equiv 1 \, (\mathrm{mod \,\, 2}) \right\}=\lfloor \sqrt{3N+1} \rfloor-\left\lfloor \frac{\sqrt{3N+1}}{3}\right\rfloor.
\end{align}
\end{theorem}
\begin{proof}
Recall the following Appell-Lerch series representations of $\omega^{(3)}(q)$ \cite{Watson}:
\begin{align}
\omega^{(3)}(q)=\frac{1}{(q^2;q^2)_\infty}\sum_{n=0}^\infty (-1)^nq^{3n(n+1)}\frac{1+q^{2n+1}}{1-q^{2n+1}}.
\end{align}
It follows that
\begin{align}
\omega^{(3)}(q)\equiv \frac{1}{(q^2;q^2)}\sum_{n=0}^\infty q^{3n(n+1)} \equiv \frac{(q^6;q^6)_\infty^3}{(q^2;q^2)_\infty}\pmod{2}.
\end{align}
The theorem follows by applying the arguments in the proof of Theorem \ref{thm-mock-3-nu}.
\end{proof}

To study the parity of the mock theta function $\rho^{(3)}(q)$, we need some preparations.
\begin{lemma}\label{known-id}
We have
\begin{align}
\frac{1}{J_1^2}&=\frac{J_8^5}{J_2^5J_{16}^2}+2q\frac{J_4^2J_{16}^2}{J_2^5J_8}, \label{J1-square} \\
\frac{J_1}{J_3^3}&=\frac{J_2J_4^2J_{12}^2}{J_6^7}-q\frac{J_2^3J_{12}^6}{J_4^2J_6^9}, \label{J1J33} \\
\frac{J_3^3}{J_1}&=\frac{J_4^3J_6^2}{J_2^2J_{12}}+q\frac{J_{12}^3}{J_4}, \label{J33J1} \\
\frac{J_3}{J_1^3}&=\frac{J_4^6J_6^3}{J_2^9J_{12}^2}+3q\frac{J_4^2J_6J_{12}^2}{J_2^7}. \label{J3J13}
\end{align}
\end{lemma}
\begin{proof}
The identities \eqref{J1-square}, \eqref{J33J1} and \eqref{J3J13} can be found in  \cite[Eqs.\ (2.15), (3.75) and (3.78)]{Xia-Yao}. The identity \eqref{J1J33} can be found in  \cite{HGB}.
\end{proof}
We also need another identity, which appears to be new and we state it separately.
\begin{lemma}\label{lem-J13}
We have
\begin{align}
\frac{J_1^2}{J_3^2}=\frac{J_2J_4^2J_{12}^4}{J_6^5J_8J_{24}}-2q\frac{J_2^2J_8J_{12}J_{24}}{J_4J_6^4}. \label{J13-2dissection}
\end{align}
\end{lemma}
\begin{proof}
From \cite[p.\ 61, Exercise 2.16]{Gasper-Rahman} we find
\begin{align}
[x\lambda, x/\lambda, \mu v, \mu/v;q]_\infty =[xv,x/v,\lambda \mu, \mu/\lambda;q]_\infty +\frac{\mu}{\lambda}[x\mu,x/\mu,\lambda v,\lambda/v;q]_\infty. \label{key-dissection}
\end{align}
Taking $(x,\lambda, \mu, v, q)\rightarrow (q,1,-q,i,q^6)$ in \eqref{key-dissection}, we deduce that
\begin{align}
[q,q,-iq,iq;q^6]_\infty =[iq,-iq,-q,-q;q^6]_\infty -q[-q^2,-1,i,-i;q^6]_\infty. \label{dis-1}
\end{align}
Note that
\begin{align}
F(q):=\frac{J_1^2}{J_3^2}=(q,q^2;q^3)_\infty^2=[q,q;q^6]_\infty (q^2,q^4;q^6)_\infty^2. \label{dis-1-start}
\end{align}
By \eqref{dis-1} and \eqref{dis-1-start} we deduce that
\begin{align}
F(q)-F(-q)=-4q\frac{J_2^2J_8J_{12}J_{24}}{J_4J_6^4}. \label{F-odd}
\end{align}
Similarly, taking $(x,\lambda,\mu, v, q)\rightarrow (q^8,q^3,-q^3,-q^2,q^{12})$ in \eqref{key-dissection}, we obtain
\begin{align}
[q^{11},q^5,q^5,q;q^{12}]_\infty =[-q^{10},-q^6,-q^6,-1;q^{12}]_\infty -[-q^{11},-q^5,-q^5,-q;q^{12}]_\infty. \label{dis-2}
\end{align}
Note that
\begin{align}
\frac{J_1^2}{J_3^2}=(q,q^2;q^3)_\infty^2=[q^{11},q^5,q^5,q;q^{12}]_\infty[q^2,q^4;q^{12}]_\infty^2. \label{dis-2-start}
\end{align}
By \eqref{dis-2} and \eqref{dis-2-start} we deduce that
\begin{align}
F(q)+F(-q)=2\frac{J_2J_4^2J_{12}^4}{J_6^5J_8J_{24}}. \label{F-even}
\end{align}
From \eqref{F-odd} and \eqref{F-even} we get \eqref{J13-2dissection} immediately.
\end{proof}
\begin{theorem}
We have
\begin{align}
\sum_{n=0}^\infty c(\rho^{(3)};4n)q^n&=\frac{J_4J_6^3}{J_{3}^2J_{12}},  \label{rho-4n}\\
\sum_{n=0}^\infty c(\rho^{(3)};4n+1)q^n&=-\frac{J_2^3J_{12}}{J_1J_3J_4}, \label{rho-4n1} \\
\sum_{n=0}^\infty c(\rho^{(3)};4n+2)q^n&=q^{-1}m(-q,q^6,q)+\frac{J_2^6J_{12}^3}{J_1^2J_4^3J_6^3}, \label{rho-4n2} \\
\sum_{n=0}^\infty c(\rho^{(3)};4n+3)q^n&=q^{-1}m(-q^2,q^6,q)+2\frac{J_4J_{12}^3}{J_6^3}. \label{rho-4n3}
\end{align}
\end{theorem}
\begin{proof}
It is known that \cite[Eq.\ (5.10)]{Hickerson-Mortenson}
\begin{align}\label{rho3-m}
\rho^{(3)}(q)=q^{-1}m(q,q^6,-q).
\end{align}
Taking $(x,q,z)\rightarrow (q,q^6,-q)$ in Lemma \ref{lem-m-add}, we deduce that
\begin{align}\label{rho-start}
\rho^{(3)}(q)=q^{-1}m(-q^8,q^{24},q^4)-q^{-6}m(-q^{-4},q^{24},q^4)-\frac{1}{2}q^{-2}\frac{J_2^2J_6^3J_8J_{12}^2}{J_3^2J_4^2J_{24}^3}.
\end{align}

Substituting \eqref{J1-square} with $q$ replaced by $q^3$ into \eqref{rho-start}, we deduce that
\begin{align}
\sum_{n=0}^\infty c(\rho^{(3)};2n)q^n=-q^{-3}m(-q^{-2},q^{12},q^2)-\frac{1}{2}q^{-1}\frac{J_1^2J_4J_6^2J_{12}^2}{J_2^2J_3^2J_{24}^2}, \label{rho-2n} \\
\sum_{n=0}^\infty c(\rho^{(3)};2n+1)q^n=q^{-1}m(-q^4,q^{12},q^2)-\frac{J_1^2J_4J_6^4J_{24}^2}{J_2^2J_3^2J_{12}^4}. \label{rho-2n1}
\end{align}
Now we substitute \eqref{J13-2dissection} into \eqref{rho-2n}. If we extract the terms in which the power of $q$ is even,  we get \eqref{rho-4n}. If we extract the terms in which the power of $q$ is odd, we get
\begin{align}
\sum_{n=0}^\infty c(\rho^{(3)};4n+2)q^n=q^{-1}m(-q,q^6,q^5)-\frac{1}{2}q^{-1}\frac{J_2^3J_6^6}{J_1J_3^3J_4J_{12}^3}. \label{add-rho4n2}
\end{align}
Here we used the fact that $m(-q^{-1},q^6,q)=-qm(-q,q^6,q^5)$, which follows from \cite[Eqs.\ (3.2a) and (3.2b)]{Hickerson-Mortenson}. Next, taking $(x,q,z_0,z_1)\rightarrow (-q,q^6,q,q^5)$ in \cite[Eq.\ (3.7)]{Hickerson-Mortenson}, we deduce that
\begin{align}
m(-q,q^6,q^5)-m(-q,q^6,q)=\frac{1}{2}\frac{J_2^6J_3^3}{J_1^3J_4^2J_6^3}. \label{add-rho-4n2-m}
\end{align}
Substituting this into \eqref{add-rho4n2}, we obtain
\begin{align}
\sum_{n=0}^\infty c(\rho^{(3)};4n+2)q^n=q^{-1}m(-q,q^6,q)+\frac{1}{2}q^{-1}\frac{1}{J_1^2}\left(\frac{J_2^6J_3^3}{J_1J_4^2J_6^3}-\frac{J_1J_2^3J_6^6}{J_3^3J_4J_{12}^3} \right). \label{rho-new-help}
\end{align}
Substituting \eqref{J1J33} and \eqref{J33J1} into \eqref{rho-new-help}, after simplifications, we arrive at \eqref{rho-4n2}.

Similarly, substituting \eqref{J13-2dissection} into \eqref{rho-2n1} and extracting the terms in which the power of $q$ is even (resp.\ odd), we get \eqref{rho-4n1} (resp.\ \eqref{rho-4n3}).
\end{proof}

\begin{theorem}\label{thm-rho3-rec}
We have
\begin{align}\label{rho3-rec}
c(\rho^{(3)};4n+2)+c(\rho^{(3)};n)\equiv a_3(n/2) \pmod{2}
\end{align}
and
\begin{align}\label{rho3-16n2}
\sum_{n=0}^\infty c(\rho^{(3)};16n+2)q^n=2q\frac{J_1J_6^3J_{12}^2}{J_3^5}.
\end{align}
\end{theorem}
\begin{proof}
From \eqref{rho3-m}, \eqref{rho-4n2} and the binomial theorem, we deduce that
\begin{align}
&\sum_{n=0}^\infty \left(c(\rho^{(3)};4n+2)+(-1)^nc(\rho^{(3)};n)\right)q^n=\frac{J_2^6J_{12}^3}{J_1^2J_4^3J_6^3} \label{rho3-add} \\
\equiv  & \frac{J_6^3}{J_2}=\sum_{n=0}^\infty a_3(n)q^{2n} \pmod{2}. \nonumber
\end{align}
This gives \eqref{rho3-rec}.

Substituting \eqref{J1-square} into \eqref{rho3-add} and extracting the terms in which the power of $q$ is even, we obtain
\begin{align}
\sum_{n=0}^\infty \left(c(\rho^{(3)};8n+2)+c(\rho^{(3)};2n)q^n \right)q^n=\frac{J_1J_4^5J_6^3}{J_2^3J_3^3J_8^2}. \label{rho-add-middle}
\end{align}

Substituting \eqref{J1J33} into \eqref{rho-add-middle} and extracting the terms in which the power of $q$ is even, we obtain
\begin{align}
\sum_{n=0}^\infty \left(c(\rho^{(3)};16n+2)+c(\rho^{(3)};4n)\right)q^n=\frac{J_2^7J_6^2}{J_1^2J_3^4J_4^2}. \label{rho-add-last}
\end{align}
Substituting \eqref{rho-4n} into \eqref{rho-add-last}, we see that to prove \eqref{rho3-16n2}, it suffices to show that
\begin{align}
\frac{J_2^7J_3}{J_1^3J_4^2}-\frac{J_4J_6J_3^3}{J_1J_{12}}=2qJ_6J_{12}^2. \label{rho-last}
\end{align}
Substituting \eqref{J33J1} and \eqref{J3J13} into the left side of \eqref{rho-last}, we get its right side immediately. Thus \eqref{rho-last} holds and we proved \eqref{rho3-16n2}.
\end{proof}

\begin{theorem}\label{thm-mock-3-rho-4n}
The coefficient $c(\rho^{(3)};4n)$ is odd if and only if $n=2k(3k+1)$ for some integer $k$.
\end{theorem}
\begin{proof}
By \eqref{rho-4n} and the binomial theorem, we deduce that
\begin{align}
\sum_{n=0}^\infty c(\rho^{(3)};4n)q^n\equiv J_4 =\sum_{k=-\infty}^\infty q^{2k(3k+1)} \pmod{2}.
\end{align}
The theorem follows immediately.
\end{proof}
\begin{theorem}\label{thm-mock-3-rho-2n1}
The coefficient $c(\rho^{(3)};2n+1)$ is odd if and only if $6n+5=p^{4a+1}m^2$ for some prime $p$ and integer $m$ with $p\nmid m$. Moreover,
\begin{align}\label{rho-2n1-estimate}
\#\left\{0\leq n\leq N:c(\rho^{(3)};2n+1)\equiv 1 \text{ \rm{(mod  2)}} \right\}=\frac{\pi^2}{3}\frac{N}{\log N}+O\left( \frac{N}{\log^2 N} \right).
\end{align}
\end{theorem}
\begin{proof}
From \eqref{rho-4n1} we deduce that
\begin{align}
\sum_{n=0}^\infty c(\rho^{(3)};4n+1)q^n\equiv J_1J_3^3&=\sum_{k=-\infty}^\infty \sum_{m=0}^\infty (-1)^{k+m}(2m+1)q^{k(3k+1)/2+3m(m+1)/2} \nonumber \\
&\equiv \sum_{k=-\infty}^\infty \sum_{m=0}^\infty q^{k(3k+1)/2+3m(m+1)/2} \pmod{2}.
\end{align}
Therefore, we have
\begin{align}
\sum_{n=0}^\infty c(\rho^{(3)};4n+1)q^{24n+10} \equiv \sum_{k=-\infty}^\infty \sum_{m=0}^\infty q^{(6k+1)^2+9(2m+1)^2} \pmod{2}. \label{rho-4n1-mod2}
\end{align}
For any integer solution $(x,y)$ of $24n+10=x^2+y^2$, both $x$ and $y$ must be odd and exactly one of them is divisible by 3. For example, if $3|y$, then we can write $x=6k\pm 1$ and $y=3(2m+1)$ for $k,m\in \mathbb{Z}$. This fact together with \eqref{rho-4n1-mod2} implies
\begin{align}\label{rho-4n1-reduce}
c(\rho^{(3)};4n+1)\equiv \frac{1}{8} r_2(24n+10) \pmod{2},
\end{align}
where $r_2(n)$ denotes the number of representations of $n$ as $x^2+y^2$ with $x,y \in \mathbb{Z}$. It is known that if $n$ has the prime factorization $n=2^ap_1^{e_1}\cdots p_j^{e_j}q_1^{f_1}\cdots q_k^{f_k}$, where the primes $p_i\equiv 1$ (mod 4) and $q_i\equiv 3$ (mod 4), then (see \cite[Theorem 4.12]{Sally}, for example)
\begin{align}\label{r2n}
r_2(n)=\left\{\begin{array}{ll}
0 & \text{if some $f_i$ is odd}; \\
4(e_1+1)\cdots (e_j+1) & \text{otherwise}.
\end{array}\right.
\end{align}
Now let  $12n+5$ have the prime factorization $12n+5=2^ap_1^{e_1}\cdots p_j^{e_j}q_1^{f_1}\cdots q_k^{f_k}$ where the primes $p_i\equiv 1$ (mod 4) and $q_i\equiv 3$ (mod 4). From \eqref{rho-4n1-reduce} and \eqref{r2n} we conclude that  $c(\rho^{(3)};4n+1)$ is odd if and only if all the $e_i$ and $f_i$ are even except for one $e_s\equiv 1$ (mod 4). This proves that $c(\rho^{(3)};4n+1)$ is odd if and only if $12n+5=p^{4a+1}m^2$ for some prime $p$ and integer $m$ with $p\nmid m$.

Next, we recall from \cite[Eq.\ (5.26)]{Hickerson-Mortenson} that the sixth order mock theta function $\sigma^{(6)}(q)$ (see Section \ref{sec-mock-6}) satisfies
\begin{align}\label{sigma-6-m}
\sigma^{(6)}(q)=-m(q^2,q^6,q).
\end{align}
Comparing \eqref{rho-4n3} with \eqref{sigma-6-m} we conclude that
\begin{align}\label{rho3-sigma6}
c(\rho^{(3)};4n+3)\equiv c(\sigma^{(6)};n+1) \pmod{2}.
\end{align}
By Theorem \ref{thm-mock-6-sigma}, which we will prove in Section \ref{sec-mock-6}, we deduce that $c(\rho^{(3)};4n+3)$ is odd if and only if $12n+11=p^{4a+1}m^2$ for some prime $p$ and integer $m$ with $p\nmid m$.

From the above discussions we get the first assertion of the theorem. Taking $(A,B)=(6,5)$ in Lemma \ref{lem-estimate}, we get \eqref{rho-2n1-estimate}.
\end{proof}

\begin{theorem}\label{thm-3-rho-almost}
The coefficient $c(\rho^{(3)};n)$ takes even values almost all the time.
\end{theorem}
\begin{proof}
Let $s_{2k}=s_{2k+1}=(2^{2k+1}-2)/3$ for $k\geq 0$. We claim that given $k\geq 1$ and $0\leq s <2^k$, if $s\neq s_k$, then $c(\rho^{(3)};2^kn+s_k)$ is even almost all the time.

Indeed, from Theorems \ref{thm-mock-3-rho-4n} and \ref{thm-mock-3-rho-2n1} we know that both $c(\rho^{(3)};2n+1)$ and $c(\rho^{(3)};4n)$ are even almost all the time. Hence the claim holds for $k=1$ and $k=2$. Now suppose that the claim holds for all $k\leq K$ with $K\geq 2$. We are going to show that the claim also holds for $k=K+1$.

If the sequence $c(\rho^{(3)};2^{K+1}n+s)$ is not almost always even, then by assumption we must have $s=s_K$ or $s=s_K+2^K$.

{\bf Case 1:} $K=2m$ is even.
From \eqref{3-core-formula} we know that $a_3(n)$ is even almost all the time. Since $c(\rho^{(3)}; 2^{2m-1}n+(5\cdot 2^{2m-2}-2)/3)$ is even almost all the time, replacing $n$ by $2^{2m-1}n+(5\cdot 2^{2m-2}-2)/3$ in \eqref{rho3-rec}, we deduce that  $c(\rho^{(3)}; 2^{2m+1}n+(5\cdot 2^{2m}-2)/3)$ is even almost all the time. Hence we must have $s=s_K=s_{K+1}$.

{\bf Case 2:} $K=2m-1$ is odd. Since $c(\rho^{(3)}; 2^{2m-2}n+(2^{2m-3}-2)/3)$ is even almost all the time, replacing $n$ by $2^{2m-2}n+(2^{2m-3}-2)/3$ in \eqref{rho3-rec}, we deduce that  $c(\rho^{(3)}; 2^{2m}n+(2^{2m-1}-2)/3)$ is even almost all the time. Hence we must have $s=s_K+2^K=s_{K+1}$.

By mathematical induction, we have proved our claim.
Note that the claim shows that for any $k\geq 1$,
\begin{align}
\frac{\#\left\{0\leq n< N:c(\rho^{(3)};n)\equiv 1 \text{ \rm{(mod  2)}} \right\}}{N} \leq \frac{1}{2^k}+\frac{1}{N}.
\end{align}
Letting $k,N\rightarrow \infty$, we see that $c(\rho^{(3)};n)$ is even almost all the time.
\end{proof}

There is one mock theta function of order 3 left: $\chi^{(3)}(q)$. Numerical evidence suggests that $c(\chi^{(3)};n)$ takes odd values half of the time (see Table \ref{tab-density} in Section \ref{sec-concluding}).

\section{Mock theta functions of order 5}\label{sec-mock-5}
In his last letter to Hardy, Ramanujan gave ten mock theta functions of order 5:
\begin{align*}
&f_0^{(5)}(q):=\sum_{n= 0}^{\infty} \frac{q^{n^2}}{(-q;q)_{n}}, 
\quad \phi_{0}^{(5)}(q):=\sum_{n=0}^{\infty}q^{n^2}(-q;q^2)_{n}, \\ 
& \psi_0^{(5)}(q):=\sum_{n=0}^\infty q^{(n+2)(n+1)/2}(-q;q)_n,  
\quad F_0^{(5)}(q):=\sum_{n= 0}^{\infty} \frac{q^{2n^2}}{(q;q^2)_{n}},  
\quad f_1^{(5)}(q):=\sum_{n=0}^{\infty} \frac{q^{n(n+1)}}{(-q;q)_{n}},\\ 
& \phi_1^{(5)}(q):=\sum_{n=0}^{\infty} q^{(n+1)^2}(-q;q^2)_{n},  
\quad \psi_1^{(5)}(q):=\sum_{n=0}^{\infty}q^{\binom{n+1}{2}}(-q;q)_{n}, \\ 
& F_1^{(5)}(q):=\sum_{n=0}^{\infty} \frac{q^{2n(n+1)}}{(q;q^2)_{n+1}},  
\quad \chi_0^{(5)}(q):=\sum_{n=0}^{\infty} \frac{q^n}{(q^{n+1};q)_{n}}, 
\quad \chi_1^{(5)}(q)=\sum_{n=0}^{\infty} \frac{q^n}{(q^{n+1};q)_{n+1}}.   
\end{align*}

\begin{theorem}\label{thm-mock-5-f0-f1}
The coefficient $c(f_0^{(5)};2n+1)$ is odd if and only if $120n+59=p^{4a+1}m^2$, where $a\geq 0$, $p$ is a prime and $p\nmid m$. The coefficient $c(f_1^{(5)};2n)$ is odd if and only if $120n+11=p^{4a+1}m^2$, where $a\geq 0$, $p$ is a prime and $p\nmid m$. Moreover,
\begin{align}
\#\left\{0\leq n\leq N: c(f_0^{(5)};2n+1)\equiv 1 \text{ \rm{(mod  2)}} \right\}=\frac{2\pi^2}{5}\frac{N}{\log N}+O\left( \frac{N}{\log^2 N} \right), \label{mock-5-f0-1} \\
\#\left\{n\leq N: c(f_1^{(5)};2n)\equiv 1 \, (\mathrm{mod \,\, 2}) \right\}=\frac{2\pi^2}{5}\frac{N}{\log N}+O\left( \frac{N}{\log^2 N} \right). \label{mock-5-f1-1}
\end{align}
\end{theorem}
\begin{proof}
From definition we have
\begin{align}
f_0^{(5)}(q)=\sum_{n=0}^\infty \frac{q^{n^2}}{(-q;q)_n}\equiv \sum_{n=0}^\infty \frac{q^{n^2}}{(q;q)_n}\equiv G(q) \pmod{2}, \\
f_1^{(5)}(q)=\sum_{n=0}^\infty \frac{q^{n(n+1)}}{(-q;q)_n} \equiv \sum_{n=0}^\infty \frac{q^{n(n+1)}}{(q;q)_n} \equiv H(q) \pmod{2}.
\end{align}
The first half of the theorem then follows from the work of Gordon \cite{Gordon}, which we have mentioned in Section \ref{sec-intro}. Applying Lemma \ref{lem-estimate} with $(A,B)=(120,59)$ and $(120,11)$, we get \eqref{mock-5-f0-1} and \eqref{mock-5-f1-1}, respectively.
\end{proof}
Using \cite[Theorem 1.3]{Chen} we can deduce that
\begin{align}
\#\left\{n\leq N: c(f_0^{(5)};2n)\equiv 1 \text{ \rm{(mod 2)}} \right\}\gg \frac{\sqrt{N}}{\log\log N},
\label{mock-5-f0-2} \\
\#\left\{n\leq N: c(f_1^{(5)};2n+1)\equiv 1 \, (\mathrm{mod \,\, 2}) \right\}\gg \frac{\sqrt{N}}{\log\log N}. \label{mock-5-f1-2}
\end{align}
As before, since these estimates are direct consequence of known results and are far from proving Conjecture \ref{conj-mock}, we do not state them in our theorems.

We remak that we have improved the error terms in Chen's results \eqref{intro-Chen-G} and \eqref{intro-Chen-H}.
\begin{corollary}\label{cor-Chen}
For sufficiently large $N$,
\begin{align}
\#\{n\leq N:  c(G;2n+1)\equiv 1 \,\, \mathrm{(mod \,\, 2)}\}
=\frac{2\pi^2}{5}\frac{N}{\log N}+O\left(\frac{N}{\log^2N} \right),  \label{Chen-G-improve} \\
\#\{n\leq N: c(H;2n)\equiv 1 \,\, \mathrm{(mod \,\, 2)}\}=\frac{2\pi^2}{5}\frac{N}{\log N}+O\left(\frac{N}{\log^2N} \right). \label{Chen-H-improve}
\end{align}
\end{corollary}

\begin{theorem}\label{thm-mock-5-psi0}
The coefficient $c(\psi_0^{(5)};n)$ is odd if and only if $60n-1=p^{4a+1}m^2$ for some prime $p$ and integer $m$ with $p\nmid m$. Moreover,
\begin{align}\label{mock-5-psi0-estimate}
\#\{n\leq N: c(\psi_0^{(5)};n)\equiv 1 \,\, \mathrm{(mod \,\, 2)}\}
=\frac{2\pi^2}{5}\frac{N}{\log N}+O\left(\frac{N}{\log^2N} \right).
\end{align}
\end{theorem}
\begin{proof}
Andrews \cite{Andrews-TAMS} found the following Hecke-type series representation:
\begin{align}
\psi_0^{(5)}(q)=-\frac{(-q;q)_\infty}{(q;q)_\infty}\sum_{n=1}^\infty \sum_{j=-n}^{n-1} (-1)^j(1-q^n)q^{n(5n-1)/2-j(3j+1)/2}.
\end{align}
We deduce that
\begin{align}
\psi_0^{(5)}(q)\equiv \sum_{n=1}^\infty \sum_{j=-n}^{n-1}(1+q^n)q^{n(5n-1)/2-j(3j+1)/2} \pmod{2}.
\end{align}
This implies
\begin{align}
&\sum_{n=1}^\infty c(\psi_0^{(5)};n)q^{6(60n-1)}\equiv \sum_{n=1}^\infty \sum_{j=-n}^{n-1} \left(q^{(30n-3)^2-15(6j+1)^2} +q^{(30n+3)^2-15(6j+1)^2}\right) \nonumber \\
\equiv & \frac{1}{2}\sum_{n=1}^\infty \sum_{j=-n}^{n-1}  \left(q^{(30n-3)^2-15(6j+1)^2}+ q^{(30n-3)^2-15(6j+5)^2} \right. \nonumber \\ &\quad \left. +q^{(30n+3)^2-15(6j+1)^2}+q^{(30n+3)^2-15(6j+5)^2}\right) \pmod{2}. \label{5-psi0-key}
\end{align}
The fundamental solution of $x^2-15y^2=1$ is $(x_1,y_1)=(4,1)$. Therefore, by Lemma \ref{lem-Pell} and \eqref{5-psi0-key} we conclude that
\begin{align}
c(\psi_0^{(5)};n)\equiv \frac{1}{2}H_{\mathbb{Z}[\sqrt{15}]}(6(60n-1)) \pmod{2}.
\end{align}
By Lemma \ref{lem-15-HT} we have $H_{\mathbb{Z}[\sqrt{15}]}(6(60n-1))=T_{\mathbb{Z}[\sqrt{15}]}(6(60n-1))$. By Lemma \ref{lem-Tn-15} it is easy to see that $\frac{1}{2}H_{\mathbb{Z}[\sqrt{15}]}(6(60n-1))$ is odd if and only if $60n-1=p^{4a+1}m^2$ for some prime $p$ and integer $m$ with $p\nmid m$.

Applying Lemma \ref{lem-estimate} with $(A,B)=(60,59)$, we get \eqref{mock-5-psi0-estimate}.
\end{proof}
\begin{theorem}\label{thm-mock-5-psi1}
The coefficient $c(\psi_1^{(5)};n)$ is odd if and only if $60n+11=p^{4a+1}m^2$ for some prime $p$ and integer $m$ with $p\nmid m$. Moreover,
\begin{align}\label{mock-5-psi1-estimate}
\#\{n\leq N: c(\psi_1^{(5)};n)\equiv 1 \,\, \mathrm{(mod \,\, 2)}\}
=\frac{2\pi^2}{5}\frac{N}{\log N}+O\left(\frac{N}{\log^2N} \right).
\end{align}
\end{theorem}
\begin{proof}
Andrews \cite{Andrews-TAMS} found the following Hecke-type series representation:
\begin{align}
\psi_1^{(5)}(q)=\frac{(-q;q)_\infty}{(q;q)_\infty}\sum_{n=0}^\infty \sum_{j=-n}^{n} (-1)^j(1-q^{2n+1})q^{n(5n+3)/2-j(3j+1)/2}.
\end{align}
We deduce that
\begin{align}
\psi_1^{(5)}(q)\equiv \sum_{n=0}^\infty \sum_{j=-n}^{n}(1+q^{2n+1})q^{n(5n+3)/2-j(3j+1)/2} \pmod{2}.
\end{align}
This implies
\begin{align}
&\sum_{n=0}^\infty c(\psi_1^{(5)};n)q^{6(60n+11)}\equiv \sum_{n=0}^\infty \sum_{j=-n}^{n} \left(q^{(30n+9)^2-15(6j+1)^2} +q^{(30n+21)^2-15(6j+1)^2}\right) \nonumber \\
\equiv & \frac{1}{2}\sum_{n=0}^\infty \sum_{j=-n}^{n}  \left(q^{(30n+9)^2-15(6j+1)^2}+ q^{(30n+9)^2-15(6j-1)^2} \right. \nonumber \\ &\quad \left. +q^{(30n+21)^2-15(6j+1)^2}+q^{(30n+21)^2-15(6j-1)^2}\right) \pmod{2}. \label{5-psi1-key}
\end{align}
By Lemma \ref{lem-Pell} and \eqref{5-psi1-key} we conclude that
\begin{align}
c(\psi_1^{(5)};n)\equiv \frac{1}{2}H_{\mathbb{Z}[\sqrt{15}]}(6(60n+11)) \pmod{2}.
\end{align}
By Lemma \ref{lem-15-HT} we have $H_{\mathbb{Z}[\sqrt{15}]}(6(60n+11))=T_{\mathbb{Z}[\sqrt{15}]}(6(60n+11))$. By Lemma \ref{lem-Tn-15} it is easy to see that $\frac{1}{2}H_{\mathbb{Z}[\sqrt{15}]}(6(60n+11))$ is odd if and only if $60n+11=p^{4a+1}m^2$ for some prime $p$ and integer $m$ with $p\nmid m$.

Applying Lemma \ref{lem-estimate} with $(A,B)=(60,11)$, we get \eqref{mock-5-psi1-estimate}.
\end{proof}
The parity properties for the fifth order mock theta functions $F_0^{(5)}(q)$ and $F_1^{(5)}(q)$ follow as direct consequences of Theorems \ref{thm-mock-5-psi0} and \ref{thm-mock-5-psi1}. This is because of the following result.
\begin{theorem}\label{thm-mock-5-F0}
For $n\geq 1$, we have
\begin{align}
c(F_0^{(5)};n)&=c(\psi_0^{(5)};2n), \label{psi0-F0-equal}\\
c(F_1^{(5)};n)&=c(\psi_1^{(5)};2n+1). \label{psi1-F1-equal}
\end{align}
\end{theorem}
\begin{proof}
Recall one of Ramanujan's theta functions:
\begin{align}
\psi(q)=\sum_{n=0}^\infty q^{n(n+1)/2}=\frac{(q^2;q^2)_\infty^2}{(q;q)_\infty}. \label{psi-defn}
\end{align}
From \cite[Entries 3.4.4 and 3.4.8]{lost-notebook5} we know that
\begin{align}
\psi_0^{(5)}(q)-F_0(q^2)+1=q\psi(q^2)H(q^4), \label{psi0-F0} \\
\psi_1^{(5)}(q)-qF_1^{(5)}(q^2)=\psi(q^2)G(q^4). \label{psi1-F1}
\end{align}
Comparing the coefficients of $q^{2n}$ (resp.\ $q^{2n+1}$) on both sides of \eqref{psi0-F0} (resp.\ \eqref{psi1-F1}), we obtain  \eqref{psi0-F0-equal} (resp.\ \eqref{psi1-F1-equal}).
\end{proof}

There are four mock theta functions of order 5 left: $\phi_0^{(5)}(q)$, $\phi_1^{(5)}(q)$, $\chi_0^{(5)}(q)$ and $\chi_1^{(5)}(q)$.  From \cite[Entries 3.4.9 and 3.4.11]{lost-notebook5} we know that
\begin{align}
\chi_0^{(5)}(q)=2F_0^{(5)}(q)-\phi_0^{(5)}(-q), \quad \chi_1^{(5)}(q)=2F_1^{(5)}(q)+q^{-1}\phi_1^{(5)}(-q).
\end{align}
Therefore, we have
\begin{align}\label{5-chi-phi}
c(\chi_0^{(5)};n)\equiv c(\phi_0^{(5)};n) \pmod{2}, \quad c(\chi_1^{(5)};n)\equiv c(\phi_1^{(5)};n+1) \pmod{2}.
\end{align}
From the above congruence relations, we know that there are essentially two functions left: $\phi_0^{(5)}(q)$ and $\phi_1^{(5)}(q)$. Numerical evidence suggests that they are all of type $(\frac{1}{2},\frac{1}{2})$ modulo 2 (see Table \ref{tab-density} in Section \ref{sec-concluding}).

\section{Mock theta function of order 6}\label{sec-mock-6}
Ramanujan recorded seven mock theta functions of order 6 in his lost notebook \cite{lostnotebook}:
\begin{align*}
&\phi^{(6)}(q):=\sum_{n=0}^{\infty}\frac{(-1)^n(q;q^2)_{n}q^{n^2}}{(-q;q)_{2n}},  
\quad \psi^{(6)}(q):=\sum_{n=0}^{\infty}\frac{(-1)^nq^{(n+1)^2}(q;q^2)_{n}}{(-q;q)_{2n+1}}, \\  
& \rho^{(6)}(q):=\sum_{n=0}^{\infty}\frac{q^{\binom{n+1}{2}}(-q;q)_{n}}{(q;q^2)_{n+1}},  
\quad \sigma^{(6)}(q):=\sum_{n=0}^{\infty}\frac{q^{\binom{n+2}{2}}(-q;q)_{n}}{(q;q^2)_{n+1}}, \\  
& \lambda^{(6)}(q):=\sum_{n=0}^{\infty}\frac{(-1)^nq^n(q;q^2)_{n}}{(-q;q)_{n}},  
\quad \mu^{(6)}(q):=\frac{1}{2}+\frac{1}{2}\sum_{n=0}^{\infty}\frac{(-1)^nq^{n+1}(1+q^n)(q;q^2)_{n}}{(-q;q)_{n+1}}, \\ 
& \gamma^{(6)}(q):=\sum_{n=0}^{\infty}\frac{q^{n^2}(q;q)_{n}}{(q^3;q^3)_{n}}.   
\end{align*}
The definition of $\mu^{(6)}(q)$ has been corrected since in Ramanujan's original definition, the series does not converge.

In 2007, after an examination of the summands of $\phi^{(6)}(q)$ and $\psi^{(6)}(q)$ over the negative and non-positive integers, Berndt and Chan \cite{Berndt-Chan} introduced two new mock theta functions of order 6:
\begin{align*}
\phi_{-}^{(6)}(q):=\sum_{n=1}^{\infty}\frac{q^n(-q;q)_{2n-1}}{(q;q^2)_{n}}, 
\quad \psi_{-}^{(6)}(q):=\sum_{n=1}^{\infty}\frac{q^{n}(-q;q)_{2n-2}}{(q;q^2)_{n}}.
\end{align*}

\begin{theorem}\label{thm-mock-6-phi}
For $n\geq 0$ we have $c(\phi^{(6)};n)\equiv p(n)$ \text{\rm{(mod 2)}}.
\end{theorem}
\begin{proof}
By definition we have
\begin{align}
\phi^{(6)}(q)=\sum_{n=0}^\infty \frac{(-1)^n(q;q^2)_nq^{n^2}}{(-q;q)_{2n}}\equiv \sum_{n=0}^\infty \frac{q^{n^2}}{(q;q)_n^2}=\sum_{n=0}^\infty p(n)q^n \pmod{2}.
\end{align}
Therefore, $c(\phi^{(6)};n)\equiv p(n)$ (mod 2).
\end{proof}

\begin{theorem}\label{thm-mock-6-psi}
We have
\begin{align}\label{mock-6-psi}
\psi^{(6)}(q)\equiv q \frac{(q^3;q^3)_\infty^6}{(q;q)_\infty^3} \pmod{2}.
\end{align}
\end{theorem}
\begin{proof}
Andrews and Hickerson \cite{Andrews-Hickerson} found that:
\begin{align}
\psi^{(6)}(q)=q\frac{(q;q^2)_\infty}{(q^2;q^2)_\infty}\sum_{n=0}^\infty \sum_{j=-n}^n (-1)^{n+j} q^{3n^2+3n-j^2}.
\end{align}
Therefore, we have
\begin{align*}
\psi^{(6)}(q)\equiv q\frac{(q;q^2)_\infty}{(q^2;q^2)_\infty} \sum_{n=0}^\infty q^{3n^2+3n} \equiv q\frac{(q;q)_\infty}{(q^2;q^2)_\infty^2}\cdot \frac{(q^{12};q^{12})_\infty^2}{(q^6;q^6)_\infty} \equiv q\frac{(q^3;q^3)_\infty^6}{(q;q)_\infty^3} \pmod{2}.
\end{align*}
\end{proof}
Numerical evidence suggests that $c(\psi^{(6)};n)$ takes odd values half of the time (see Table \ref{tab-density} in Section \ref{sec-concluding}).

\begin{theorem}\label{thm-mock-6-rho}
The coefficient $c(\rho^{(6)};n)$ is odd if and only if $n=k(k+1)$ for some integer $k\geq 0$. Moreover,
\begin{align}
\#\left\{0\leq n\leq N:c(\rho^{(6)};n)\equiv 1 \, (\mathrm{mod \,\, 2}) \right\} = \left\lfloor  \frac{\sqrt{4N+1}+1}{2}\right\rfloor.
\end{align}
\end{theorem}
\begin{proof}
 Andrews and Hickerson \cite{Andrews-Hickerson} found the following Hecke-type series representation:
\begin{align}
\rho^{(6)}(q)=\frac{(-q;q)_\infty}{(q;q)_\infty}\sum_{n=0}^\infty \sum_{j=-n}^n (-1)^nq^{3n(n+1)/2-j(j+1)/2}.
\end{align}
This implies
\begin{align}
\rho^{(6)}(q)\equiv \sum_{n=0}^\infty q^{n(n+1)} \pmod{2}.
\end{align}
The theorem follows immediately.
\end{proof}

\begin{theorem}\label{thm-mock-6-sigma}
The coefficient $c(\sigma^{(6)};n)$ is odd if and only if $12n-1=p^{4a+1}m^2$ where $p$ is a prime and $m$ is an integer with $p\nmid m$. Moreover,
\begin{align}\label{mock-6-sigma-estimate}
\#\left\{0\leq n\leq N:c(\sigma^{(6)};n)\equiv 1 \, (\mathrm{mod \,\, 2}) \right\} =\frac{\pi^2}{3} \frac{N}{\log N}+O\left(\frac{N}{\log^2 N} \right).
\end{align}
\end{theorem}
\begin{proof}
 Andrews and Hickerson \cite{Andrews-Hickerson} found the following Hecke-type series representation:
\begin{align}
\sigma^{(6)}(q)=q\frac{(-q;q)_\infty}{(q;q)_\infty}\sum_{n=0}^\infty (-1)^n(1-q^{n+1})q^{(3n^2+5n)/2}\sum_{j=0}^nq^{-j(j+1)/2}.
\end{align}
We deduce that
\begin{align}
\sum_{n=0}^\infty c(\sigma^{(6)};n)q^{24n-2}\equiv \frac{1}{2}\sum_{n=0}^\infty \sum_{j=-n-1}^n\left(q^{(6n+5)^2-3(2j+1)^2}+q^{(6n+7)^2-3(2j+1)^2}  \right) \pmod{2}.
\end{align}
The fundamental solution of $x^2-3y^2=1$ is $(x_1,y_1)=(2,1)$. By Lemma \ref{lem-Pell} we deduce that
\begin{align}
c(\sigma^{(6)};n)\equiv \frac{1}{2}H_{\mathbb{Z}[\sqrt{3}]}(24n-2) \pmod{2}.
\end{align}
Let $24n-2$ have the prime factorization  $24n-2=2p_1^{e_1}\cdots p_j^{e_j}q_1^{f_1}\cdots q_{1}^{f_1}\cdots q_k^{f_k}r_1^{g_1}\cdots r_{\ell}^{g_\ell}$ where the primes $p_i\equiv 1$ (mod 12), $q_i\equiv \pm 5$ (mod 12), $r_i\equiv 11$ (mod 12). By Lemma \ref{lem-Z3} we see that $c(\sigma^{(6)};n)$ is odd if and only if all of $e_i,f_i$ and $g_i$ are even except for one $g_s\equiv 1$ (mod 4). Therefore, $c(\sigma^{(6)};n)$ is odd if and only if $12n-1=p^{4a+1}m^2$ where $p$ is a prime and $p\nmid m$.

Applying Lemma \ref{lem-estimate} with $(A,B)=(12,11)$, we get \eqref{mock-6-sigma-estimate}.
\end{proof}

\begin{theorem}\label{thm-mock-6-lambda}
We have
\begin{align}
\sum_{n=0}^\infty c(\lambda^{(6)};2n)q^n =\frac{J_2^3J_3^2}{J_1^3J_6}, \label{6-lambda-psi-id-1} \\
\sum_{n=1}^\infty \left(2c(\psi^{(6)};n)-c(\lambda^{(6)};2n-1) \right)q^n=3q\frac{J_6^3}{J_1J_2}. \label{6-lambda-psi-id-2}
\end{align}
Furthermore, we have
\begin{align}\label{6-lambda-psi}
c(\lambda^{(6)};2n-1)\equiv c(\psi^{(6)};n) \pmod{2}.
\end{align}
The coefficient $c(\lambda^{(6)};2n)$ is odd if and only if $n=k(k+1)/2$ for some integer $k$.
\end{theorem}
%
\begin{proof}
From \cite[p.\ 140, Entry 7.5.2]{lost-notebook5} we find
\begin{align}\label{6-psi-lambda-Rama-id}
2q^{-1}\psi^{(6)}(q^2)+\lambda^{(6)}(-q)=\frac{J_2^6J_3J_{12}}{J_1^3J_4^3J_6}.
\end{align}
Substituting \eqref{J3J13} into \eqref{6-psi-lambda-Rama-id} and extracting the terms in which the power of $q$ is even (resp.\ odd), we obtain \eqref{6-lambda-psi-id-1} (resp.\ \eqref{6-lambda-psi-id-2}).

From \eqref{6-lambda-psi-id-2} we deduce that
\begin{align}
\sum_{n=1}^\infty c(\lambda^{(6)};2n-1)q^n\equiv q\frac{J_3^6}{J_1^3} \pmod{2}.
\end{align}
Comparing this with \eqref{mock-6-psi}, we get \eqref{6-lambda-psi}.

From \eqref{6-lambda-psi-id-1} we deduce that
\begin{align}
\sum_{n=0}^\infty c(\lambda^{(6)};2n)q^n\equiv \frac{(q^2;q^2)_\infty^2}{(q;q)_\infty} \pmod{2}.
\end{align}
Recall the definition of $\psi(q)$ in \eqref{psi-defn}, we get the last assertion of the theorem.
\end{proof}

Recall that numerical evidence indicates that $c(\psi^{(6)};n)$ takes odd values half of the time. If this is indeed true, then Theorem \ref{thm-mock-6-lambda} shows that $\lambda^{(6)}(q)$ is of type $(\frac{3}{4},\frac{1}{4})$ modulo 2.

\begin{theorem}\label{thm-mock-6-phi-minus}
We have
\begin{align}\label{36-psi-phi-minus}
c(\psi^{(3)};n)\equiv c(\phi_{-}^{(6)};n) \pmod{2}.
\end{align}
The coefficient $c(\phi_{-}^{(6)};n)$ is odd if and only if $24n-1=p^{4a+1}m^2$ for some prime $p$ and integer $m$ with $p\nmid m$. Moreover,
\begin{align}
\#\left\{0\leq n\leq N:c(\phi_{-}^{(6)};n)\equiv 1  \, (\mathrm{mod \,\, 2}) \right\} =\frac{\pi^2}{3} \frac{N}{\log N}+O\left(\frac{N}{\log^2 N} \right).
\end{align}
\end{theorem}
\begin{proof}
Berndt and Chan \cite{Berndt-Chan} gave the following Hecke-type series representation:
\begin{align}
\phi_{-}^{(6)}(q)=\frac{(-q;q)_\infty}{(q;q)_\infty}\sum_{n=0}^\infty (-1)^n(1-q^{2n+2})q^{3n^2+5n+1}\sum_{j=0}^n(1+q^{2j+1})q^{-2j^2-3j}.
\end{align}
We deduce that
\begin{align}
\phi_{-}^{(6)}(q)&\equiv \sum_{n=0}^\infty (1+q^{2n+2})q^{3n^2+5n+1}\sum_{j=0}^n\left(q^{-2j^2-3j}+q^{-2j^2-j+1} \right) \nonumber \\
& \equiv \sum_{n=0}^\infty \left(q^{3n^2+5n+1}+q^{3n^2+7n+3} \right)\sum_{j=-n-1}^nq^{-2j^2-3j} \pmod{2}.
\end{align}
Therefore,
\begin{align}
&\sum_{n=1}^\infty c(\phi_{-}^{(6)};n)q^{48n-2}\equiv \sum_{n=0}^\infty \sum_{j=-n-1}^n q^{(12n+10)^2-6(4j+3)^2}+q^{(12n+14)^2-6(4j+3)^2} \nonumber \\
\equiv & \frac{1}{2}\sum_{n=0}^\infty \sum_{j=-n-1}^n\left(q^{(12n+10)^2-6(4j+3)^2}+q^{(12n+14)^2-6(4j+3)^2} \right. \nonumber \\
& \quad \left. +q^{(12n+10)^2-6(4j+1)^2}+q^{(12n+14)^2-6(4j+1)^2} \right) \pmod{2}. \label{6-phi-minus-key}
\end{align}
The fundamental solution of $x^2-6y^2=1$ is $(x_1,y_1)=(5,2)$. By Lemma \ref{lem-Pell} and \eqref{6-phi-minus-key} we conclude that
\begin{align}\label{6-phi-minus-H}
c(\phi_{-}^{(6)};n)\equiv \frac{1}{2}H_{\mathbb{Z}[\sqrt{6}]}(48n-2) \pmod{2}.
\end{align}
Comparing \eqref{3-psi-H} with \eqref{6-phi-minus-H}, we obtain \eqref{36-psi-phi-minus}. The rest of the theorem follows from Theorem \ref{thm-mock-3-psi}.
\end{proof}

\begin{theorem}\label{thm-mock-6-psi-minus}
The coefficient $c(\psi_{-}^{(6)};n)$ is odd if and only if $8n-3=p^{4a+1}m^2$ for some prime $p$ and integer $m$ with $p\equiv 5$ \text{\rm{(mod 24)}} and $p\nmid m$, or $8n-3=3p^{4a+1}m^2$ for some prime $p\equiv 23$ \text{\rm{(mod 24)}} and integer $m$ with $p\nmid m$. Moreover,
\begin{align}\label{mock-6-psi-minus-estimate}
\#\left\{0\leq n\leq N|c(\psi_{-}^{(6)};n)\equiv 1 \,  (\mathrm{mod \,\, 2}) \right\} =\frac{\pi^2}{6} \frac{N}{\log N}+O\left(\frac{N}{\log^2 N} \right).
\end{align}
\end{theorem}
\begin{proof}
We need the following Hecke-type series representation found by Berndt and Chan \cite{Berndt-Chan}:
\begin{align}
\psi_{-}^{(6)}(q)=q\frac{(-q;q)_\infty}{(q;q)_\infty}\sum_{n=0}^\infty (-1)^nq^{3n^2+3n}\sum_{j=-n}^nq^{-2j^2-j}.
\end{align}
By the binomial theorem, we deduce that
\begin{align}
\psi_{-}^{(6)}(q)\equiv q\sum_{n=0}^\infty \sum_{j=-n}^nq^{3n^2+3n-2j^2-j} \pmod{2}.
\end{align}
This implies
\begin{align}
\sum_{n=0}^\infty c(\psi_{-}^{(6)};n)q^{8n-3}\equiv \frac{1}{2}\sum_{n=0}^\infty \sum_{j=-n}^n\left(q^{6(2n+1)^2-(4j+1)^2} +q^{6(2n+1)^2-(4j-1)^2} \right) \pmod{2}.
\end{align}
By Lemma \ref{lem-Pell} we deduce that
\begin{align}
c(\psi_{-}^{(6)};n)\equiv \frac{1}{2}H_{\mathbb{Z}[\sqrt{6}]}(-8n+3) \pmod{2}.
\end{align}
Let $8n-3$ have the prime factorization $8n-3=3^bp_1^{e_1}\cdots p_j^{e_j}q_1^{f_1}\cdots q_k^{f_k}r_1^{g_1}\cdots r_\ell^{g_\ell}$ where the primes $p_i\equiv \pm 7, \pm 11$ (mod 24), $q_i\equiv  1, 19$ (mod 24), and $r_i\equiv 5,23$ (mod 24). By Lemma \ref{lem-Z6}, we know that $\frac{1}{2}H_{\mathbb{Z}[\sqrt{6}]}(-8n+3)\equiv 1$ (mod 2) if and only if all of $e_i$, $f_i$ and $g_i$ are even except for one $g_s\equiv 1$ (mod 4). Note that $8n-3\equiv 5$ (mod 8).  If $b \equiv 1 $ (mod 2), then the condition becomes $8n-3=3p^{4a+1}m^2$ for some prime $p\equiv 23$ (mod 24) and integer $m$ with $p\nmid m$. If $b\equiv 0$ (mod 2), then the condition becomes $8n-3=p^{4a+1}m^2$ for some prime $p\equiv 5$ (mod 24) and integer $m$ with $p\nmid m$.

Let
\begin{align*}
\gamma_1(N)=\#\left\{0\leq n<N: 8n+5=p^{4a+1}m^2, \, \text{ $p\equiv 5$ \rm{(mod 24)} is a prime and $p\nmid m$}\right\}.
\end{align*}
To give an estimate to $\gamma_1(N)$, we just need to replace the condition $p\equiv B/m^2$ (mod 8) by $p\equiv 5$ (mod 24) in \eqref{main-term} and accordingly replace $\phi(A)$ by $\phi(24)$ in \eqref{main-estimate}. We deduce that
\begin{align}\label{gamma1}
\gamma_1(N)=\frac{1}{\phi(24)}\left(1-\frac{1}{2^2}\right)\cdot \frac{\pi^2}{6}\frac{8N}{\log N}+O\left(\frac{N}{\log^2N}\right) =\frac{\pi^2}{8}\frac{N}{\log N}+O\left(\frac{N}{\log^2N}\right).
\end{align}
Next, note that $8n-3=3p^{4a+1}m^2$ implies $3|n$. Let $n=3(n'+1)$. Then we get $8n'+7=p^{4a+1}m^2$. Let
\begin{align*}
\gamma_2(N)=\#\left\{0\leq n'<\frac{N}{3}: 8n'+7=p^{4a+1}m^2, \, \text{ $p\equiv 23$ \rm{(mod 24)} is a prime and $p\nmid m$}\right\}.
\end{align*}
Similarly as above, we deduce that
\begin{align}\label{gamma2}
\gamma_2(N)=\frac{1}{\phi(24)}\left(1-\frac{1}{2^2}\right)\cdot \frac{\pi^2}{6}\frac{8N/3}{\log N}+O\left(\frac{N}{\log^2N}\right) =\frac{\pi^2}{24}\frac{N}{\log N}+O\left(\frac{N}{\log^2N}\right).
\end{align}
Adding \eqref{gamma1} and \eqref{gamma2} up, we get \eqref{mock-6-psi-minus-estimate}.
\end{proof}

Note that
\begin{align}\label{6-mu-half}
\mu^{(6)}(q)=\frac{1}{2}+q-\frac{3}{2}q^2+2q^3-2q^4+3q^5-\frac{11}{2}q^6+7q^7-\frac{15}{2}q^8+\cdots.
\end{align}
Since $c(\mu^{(6)};n)$ is half integer,  we consider $2\mu^{(6)}(q)$ instead.
\begin{theorem}\label{thm-mock-6-mu}
We have
\begin{align}\label{6-mu-2n1}
\sum_{n=0}^\infty c(2\mu^{(6)};2n+1)q^n=2\frac{J_2^2J_6^2}{J_1^2J_3}.
\end{align}
Moreover, $c(2\mu^{(6)};2n+1)\equiv 0$  \text{\rm{(mod 2)}} and $c(2\mu^{(6)};2n)\equiv p(n)$ \text{\rm{(mod 2)}} for any $n\geq 0$.
\end{theorem}
\begin{proof}
We recall from \cite[p.\ 140, Entry 7.5.2]{lost-notebook5} that
\begin{align}\label{6-phi-mu}
2\phi^{(6)}(q^2)-2\mu^{(6)}(-q)=\frac{J_2^4J_6^5}{J_1^2J_3^2J_4^2J_{12}^2}.
\end{align}
From \cite[Eq.\ (3.12)]{Xia-Yao} we find
\begin{align}
\frac{1}{J_1J_3}=\frac{J_8^2J_{12}^5}{J_2^2J_4J_6^4J_{24}^2}+q\frac{J_4^5J_{24}^2}{J_2^4J_6^2J_8^2J_{12}}. \label{J1J3}
\end{align}
Taking square on both sides and substituting it into \eqref{6-phi-mu}, extracting the terms in which the power of $q$ is odd, we get \eqref{6-mu-2n1}.

On the other hand, it is clear from \eqref{6-phi-mu} that
\begin{align}
2\mu^{(6)}(q)\equiv \frac{1}{J_2} \pmod{2}.
\end{align}
Thus we have $c(2\mu^{(6)};2n+1)\equiv 0$ (mod 2) and $c(2\mu^{(6)};2n)\equiv p(n)$ (mod 2).
\end{proof}

There is one function left: $\gamma^{(6)}(q)$. Numerical evidence suggests that it is of type $(\frac{1}{2},\frac{1}{2})$ modulo 2 (see Table \ref{tab-density} in Section \ref{sec-concluding}).

\section{Mock theta functions of order 7}\label{sec-mock-7}
In his last letter to Hardy, Ramanujan gave three mock theta functions of order 7:
\begin{align*}
\mathcal{F}_0^{(7)}(q):=\sum_{n=0}^\infty \frac{q^{n^2}}{(q^{n+1};q)_n}, 
\quad \mathcal{F}_1^{(7)}(q):=\sum_{n=1}^\infty \frac{q^{n^2}}{(q^n;q)_n}, 
\quad \mathcal{F}_2^{(7)}(q):=\sum_{n=0}^\infty \frac{q^{n^2+n}}{(q^{n+1};q)_{n+1}}. 
\end{align*}
\begin{theorem}\label{thm-mock-7}
We have
\begin{align}
\mathcal{F}_0^{(7)}(q) \equiv \frac{(q^6,q^8,q^{14};q^{14})_\infty}{(q;q)_\infty} \pmod{2}, \label{7-F0-mod2} \\
\mathcal{F}_1^{(7)}(q)\equiv q\frac{(q^2,q^{12},q^{14};q^{14})_\infty}{(q;q)_\infty} \pmod{2}, \label{7-F1-mod2} \\
\mathcal{F}_2^{(7)}(q)\equiv \frac{(q^4,q^{10},q^{14};q^{14})_\infty}{(q;q)_\infty} \pmod{2}. \label{7-F2-mod2}
\end{align}
\end{theorem}
\begin{proof}
We use the following Hecke-type series representations found by Andrews \cite{Andrews-TAMS}:
\begin{align}
\mathcal{F}_0^{(7)}(q)=& \frac{1}{(q;q)_\infty} \left( \sum_{n=0}^\infty \sum_{|j|\leq n} q^{7n^2+n-j^2}(1-q^{12n+6}) \right. \nonumber \\
&\quad \left. -2q\sum_{n=0}^\infty \sum_{j=0}^nq^{7n^2+8n-j^2-j}(1-q^{12n+13})\right), \label{mock-7-F0}\\
\mathcal{F}_1^{(7)}(q)=& \frac{1}{(q;q)_\infty} \left(\sum_{n=0}^\infty \sum_{j=-n}^n q^{7n^2+3n-j^2}(1-q^{8n+4})  \right. \nonumber \\ &\quad \left. -2q^3\sum_{n=0}^\infty \sum_{j=0}^n q^{7n^2+10n-j^2-j}(1-q^{8n+8})  \right), \label{mock-7-F1} \\
\mathcal{F}_2^{(7)}(q)=& \frac{1}{(q;q)_\infty} \left(\sum_{n=0}^\infty \sum_{|j|\leq n}q^{7n^2+n-j^2}(1-q^{8n+3})  \right. \nonumber \\
&\quad \left. -2q^2\sum_{n=0}^\infty \sum_{j=0}^nq^{7n^2+8n-j^2-j}(1-q^{8n+7})  \right). \label{mock-7-F2}
\end{align}
From \eqref{mock-7-F0} we deduce that
\begin{align}
\mathcal{F}_0^{(7)}(q) &\equiv \frac{1}{(q;q)_\infty}\sum_{n=0}^\infty \sum_{j=-n}^n q^{7n^2+n-j^2}(1+q^{12n+6}) \nonumber \\
&\equiv \frac{1}{(q;q)_\infty}\sum_{n=0}^\infty q^{7n^2+n}(1+q^{12n+6}) \nonumber \\
&\equiv \frac{1}{(q;q)_\infty}\sum_{n=-\infty}^\infty q^{7n^2+n} \nonumber \\
&\equiv \frac{(q^6,q^8,q^{14};q^{14})_\infty}{(q;q)_\infty} \pmod{2}.
\end{align}
This proves \eqref{7-F0-mod2}. Similarly, using \eqref{mock-7-F1} and \eqref{mock-7-F2} we can prove \eqref{7-F1-mod2} and \eqref{7-F2-mod2}.
\end{proof}
If $S$ is a subset of positive integers, we denote by $p_S(n)$  the number of partitions of $n$ with parts in $S$.
\begin{corollary}
We have
\begin{align}
c(\mathcal{F}_0^{(7)};n)\equiv p_{S_0}(n) \pmod{2}, \\
c(\mathcal{F}_1^{(7)};n)\equiv p_{S_1}(n-1) \pmod{2}, \\
c(\mathcal{F}_2^{(7)};n)\equiv p_{S_2}(n) \pmod{2}.
\end{align}
where $S_0$ denotes the set of positive integers not congruent to $0$, $6$ and $8$ modulo $14$, $S_1$ denotes the set of positive integers not congruent to $0$, $2$ and $12$ modulo $14$, and $S_2$ denotes the set of positive integers not congruent to $0, 4$ and $10$ modulo $14$.
\end{corollary}
\begin{rem}
We can then apply  \cite[Theorems 5.1 and 5.2]{BYZ} to give lower bounds for the quantity $\#\{0\leq n\leq N|c(\mathcal{F}_i^{(7)};n)\equiv r\pmod{2}\}$ for $i\in \{0,1,2\}$ and $r\in \{0,1\}$. Again, since this is far from proving Conjecture \ref{conj-mock}, we will not pursue it here.
\end{rem}

\section{Mock theta functions of order 8}\label{sec-mock-8}
 Gordon and McIntosh \cite{Gordon-McIntosh} found eight mock theta functions of order 8:
\begin{align*}
&S_0^{(8)}(q):=\sum_{n=0}^{\infty}\frac{q^{n^2}(-q;q^2)_{n}}{(-q^2;q^2)_{n}},  
\quad S_1^{(8)}(q):=\sum_{n=0}^{\infty}\frac{q^{n(n+2)}(-q;q^2)_{n}}{(-q^2;q^2)_{n}},  \\ 
&T_0^{(8)}(q):=\sum_{n=0}^{\infty}\frac{q^{(n+1)(n+2)}(-q^2;q^2)_{n}}{(-q;q^2)_{n+1}},  
\quad T_1^{(8)}(q):=\sum_{n=0}^{\infty}\frac{q^{n(n+1)}(-q^2;q^2)_{n}}{(-q;q^2)_{n+1}},   \\ 
&U_0^{(8)}(q):=\sum_{n=0}^{\infty}\frac{q^{n^2}(-q;q^2)_{n}}{(-q^4;q^4)_n}=S_0^{(8)}(q^2)+qS_1^{(8)}(q^2),    \\
&U_1^{(8)}(q):=\sum_{n=0}^{\infty}\frac{q^{(n+1)^2}(-q;q^2)_n}{(-q^2;q^4)_{n+1}}=T_0^{(8)}(q^2)+qT_1^{(8)}(q^2),   \\
&V_0^{(8)}(q):=-1+2\sum_{n=0}^{\infty}\frac{q^{n^2}(-q;q^2)_{n}}{(q;q^2)_n}=-1+2\sum_{n=0}^{\infty}\frac{q^{2n^2}(-q^2;q^4)_{n}}{(q;q^2)_{2n+1}}, \\ 
&V_1^{(8)}(q):=\sum_{n=0}^{\infty}\frac{q^{(n+1)^2}(-q;q^2)_n}{(q;q^2)_{n+1}}=\sum_{n=0}^{\infty}\frac{q^{2n^2+2n+1}(-q^4;q^4)_n}{(q;q^2)_{2n+2}}=\sum_{n=0}^{\infty}\frac{q^{n+1}(-q;q)_{2n}}{(-q^2;q^4)_{n+1}}. 
\end{align*}
\begin{theorem}\label{thm-mock-8-T0}
The coefficient $c(T_0^{(8)};n)$ is odd if and only if $16n-1=p^{4a+1}m^2$ for some prime $p$ and integer $m$ with $p\nmid m$. Moreover,
\begin{align}\label{mock-8-T0-estimate}
\#\left\{n\leq N:c(T_0^{(8)};n)\equiv 1 \,  (\mathrm{mod \,\, 2}) \right\} =\frac{\pi^2}{4} \frac{N}{\log N}+O\left(\frac{N}{\log^2 N} \right).
\end{align}
\end{theorem}
\begin{proof}
We use the following Hecke-type series representation found by Srivastava \cite[Eq.\ (5.3)]{Srivastava} and Cui, Gu and Hao \cite{CGH}:
\begin{align}
T_0^{(8)}(q)=q^2\frac{(-q^2;q^2)_\infty}{(q^2;q^2)_\infty}\sum_{n=0}^\infty q^{4n^2+7n}(1-q^{2n+2})\sum_{j=-n-1}^n(-1)^jq^{-2j^2-3j}.
\end{align}
By the binomial theorem, we deduce that
\begin{align}
T_0^{(8)}(q)\equiv q^2\sum_{n=0}^\infty \left(q^{4n^2+7n}+q^{4n^2+9n+2} \right)\sum_{j=-n-1}^nq^{-2j^2-3j} \pmod{2}.
\end{align}
This implies
\begin{align}
\sum_{n=1}^\infty c(T_0^{(8)};n)q^{16n-1} &\equiv \sum_{n=0}^\infty \sum_{j=-n-1}^n\left(q^{(8n+7)^2-2(4j+3)^2}+q^{(8n+9)^2-2(4j+3)^2} \right) \nonumber \\
&\equiv \frac{1}{2}\sum_{n=0}^\infty \sum_{j=-n-1}^n\left(q^{(8n+7)^2-2(4j+1)^2}+q^{(8n+7)^2-2(4j+3)^2}\right. \nonumber \\
&\quad \left. +q^{(8n+9)^2-2(4j+1)^2}+q^{(8n+9)^2-2(4j+3)^2}  \right) \pmod{2}.
\end{align}
By Lemma \ref{lem-Pell} we deduce that
\begin{align}
c(T_0^{(8)};n)\equiv \frac{1}{2}H_{\mathbb{Z}[\sqrt{2}]}(16n-1) \pmod{2}.
\end{align}
Let $16n-1$ have the prime factorization $16n-1=p_1^{e_1}\cdots p_j^{e_j}q_1^{f_1}\cdots q_k^{f_k}$ where the primes $p_i\equiv \pm 1$ (mod 8) and $q_i\equiv \pm 3$ (mod 8). By Lemma \ref{lem-Z2} we know that $\frac{1}{2}H_{\mathbb{Z}[\sqrt{2}]}(16n-1)$ is odd if and only if all the $e_i$ and $f_i$ are even except for exactly one $e_s\equiv 1$ (mod 4). Thus, $c(T_0^{(8)};n)$ is odd if and only if $16n-1=p^{4a+1}m^2$ for some prime $p$ and integer $m$ with $p\nmid m$.

Applying Lemma \ref{lem-estimate} with $(A,B)=(16,15)$, we get \eqref{mock-8-T0-estimate}.
\end{proof}

\begin{theorem}\label{thm-mock-8-T1}
The coefficient $c(T_1^{(8)};n)$ is odd if and only if $16n+7=p^{4a+1}m^2$ for some prime $p$ and integer $m$ with $p\nmid m$. Moreover,
\begin{align}\label{mock-8-T1-estimate}
\#\left\{n\leq N:c(T_1^{(8)};n)\equiv 1 \, (\mathrm{mod \,\, 2}) \right\} =\frac{\pi^2}{4} \frac{N}{\log N}+O\left(\frac{N}{\log^2 N} \right).
\end{align}
\end{theorem}
\begin{proof}
We use the following Hecke-type series representation found by   Srivastava \cite[Eq.\ (5.4)]{Srivastava} and Cui, Gu and Hao \cite{CGH}:
\begin{align}
T_1^{(8)}(q)=\frac{(-q^2;q^2)_\infty}{(q^2;q^2)_\infty}\sum_{n=0}^\infty q^{4n^2+3n}(1-q^{2n+1})\sum_{j=-n}^n(-1)^jq^{-2j^2-j}.
\end{align}
By the binomial theorem, we deduce that
\begin{align}
T_1^{(8)}(q)\equiv \sum_{n=0}^\infty \sum_{j=-n}^n q^{4n^2+3n}(1+q^{2n+1})\sum_{j=-n}^nq^{-2j^2-j} \pmod{2}.
\end{align}
Therefore,
\begin{align}
\sum_{n=0}^\infty c(T_1^{(8)};n)q^{16n+7} & \equiv \frac{1}{2}\sum_{n=0}^\infty \sum_{j=-n}^n \left(q^{(8n+3)^2-2(4j+1)^2}+q^{(8n+5)^2-2(4j+1)^2}\right. \nonumber \\
& \quad \left.+q^{(8n+3)^2-2(4j-1)^2}+q^{(8n+5)^2-2(4j-1)^2}  \right) \pmod{2}.
\end{align}
By Lemma \ref{lem-Pell} we deduce that
\begin{align}
c(T_1^{(8)};n)\equiv \frac{1}{2}H_{\mathbb{Z}[\sqrt{2}]}(16n+7) \pmod{2}.
\end{align}
Let $16n+7$ have the prime factorization $16n+7=p_1^{e_1}\cdots p_j^{e_j}q_1^{f_1}\cdots q_k^{f_k}$, where the primes $p_i\equiv \pm 1$ (mod 8) and $q_j\equiv \pm 3$ (mod 8).  By Lemma \ref{lem-Z2} we know that $\frac{1}{2}H_{\mathbb{Z}[\sqrt{2}]}(16n+7)$ is odd if and only if all the $f_i$ and $e_i$ are even except for exactly one $e_s\equiv 1$ (mod 4). Thus, $c(T_1^{(8)};n)$ is odd if and only if $16n+7=p^{4a+1}m^2$ for some prime $p$ and integer $m$ with $p\nmid m$.

Applying Lemma \ref{lem-estimate} with $(A,B)=(16,7)$, we get \eqref{mock-8-T1-estimate}.
\end{proof}

\begin{theorem}\label{thm-mock-8-V0}
We have
\begin{align}\label{8-V0-cong}
c(V_0^{(8)};n)\equiv \left\{\begin{array}{ll}
1 \pmod{4} & n=0,\\
2 \pmod{4} & n=k^2, \\
0 \pmod{4} &\text{otherwise}. \end{array}\right.
\end{align}
\end{theorem}
\begin{proof}
By definition, we have
\begin{align}
V_0^{(8)}(q)=-1+2\sum_{n=0}^\infty \frac{q^{n^2}(-q;q^2)_n}{(q;q^2)_n} \equiv -1+2\sum_{n=0}^\infty q^{n^2} \pmod{4}.
\end{align}
Comparing the coefficients of $q^n$ on both sides, we get \eqref{8-V0-cong}.
\end{proof}

\begin{theorem}\label{thm-mock-8-V1}
The coefficient $c(V_1^{(8)};n)$ is odd if and only if $4n-1=p^{4a+1}m^2$ for some prime $p$ and integer $m$ with $p\nmid m$. Moreover,
\begin{align}\label{mock-8-V1-estimate}
\#\left\{n\leq N:c(V_1^{(8)};n)\equiv 1 \,\, (\mathrm{mod \,\, 2}) \right\} =\frac{\pi^2}{4} \frac{N}{\log N}+O\left(\frac{N}{\log^2 N} \right).
\end{align}
\end{theorem}
\begin{proof}
By definition we have
\begin{align}
V_1^{(8)}(q)=\sum_{n=0}^{\infty}\frac{q^{(n+1)^2}(-q;q^2)_n}{(q;q^2)_{n+1}}\equiv \sum_{n=1}^\infty \frac{q^{n^2}}{1-q^{2n-1}} \pmod{2}.
\end{align}
Let
\begin{align}
\sum_{n=1}^\infty a(n)q^n:=\sum_{n=1}^\infty \frac{q^{n^2}}{1-q^{2n-1}}.
\end{align}
Then we have $c(V_1^{(8)};n)\equiv a(n)$ (mod 2).

We will show that
\begin{align}
a(n)=\frac{1}{2}d(4n-1) \label{an-exp}
\end{align}
where $d(m)$ denotes the number of positive divisor of $m$. Indeed,
\begin{align*}
\sum_{n=1}^\infty a(n)q^{4n-1}&=\sum_{n=1}^\infty \frac{q^{4n^2-1}}{1-q^{8n-4}} =\sum_{n=1}^\infty q^{(2n-1)(2n+1)}\sum_{k=0}^\infty q^{4k(2n-1)} \\
&=\sum_{n=1}^\infty \sum_{k=0}^\infty q^{(2n-1)(2n+1+4k)}.
\end{align*}
Therefore, $a(n)$ is equal to the number of pairs of integers $(m,k)$ such that $m\geq 1, k\geq 0$ and  $4n-1=(2m-1)(2m+1+4k)$. This  implies  \eqref{an-exp}.

Let $4n-1$ have the prime factorization $4n-1=p_1^{e_1}\cdots p_k^{e_k}$. We have
\begin{align}
d(4n-1)=(e_1+1)\cdots (e_k+1).
\end{align}
Clearly, $\frac{1}{2}d(4n-1)$ is odd if and only all of $e_i$ are even except for exactly one $e_s\equiv 1$ (mod 4). Therefore, $c(V_1^{(8)};n)$ is odd if and only if $4n-1=p^{4a+1}m^2$ for some prime $p$ and integer $m$ with $p\nmid m$.

Applying Lemma \ref{lem-estimate} with $(A,B)=(4,3)$, we get \eqref{mock-8-V1-estimate}.
\end{proof}

From \eqref{an-exp} we get the following interesting identity.
\begin{corollary}
We have
\begin{align}
 \sum_{n=1}^\infty \frac{q^{n^2}}{1-q^{2n-1}}=\sum_{n=1}^\infty \frac{q^n}{1-q^{4n-1}}.
\end{align}
\end{corollary}

\begin{theorem}\label{thm-mock-8-U0S0S1}
We have $c(U_0^{(8)};n)\equiv p_{-3}(n)$ \text{\rm{(mod 2)}}. Moreover, we have $c(S_0^{(8)};n)\equiv p_{-3}(2n)$ \text{\rm{(mod 2)}} and $c(S_1^{(8)};n)\equiv p_{-3}(2n+1)$ \text{\rm{(mod 2)}}.
\end{theorem}
\begin{proof}
Gordon and McIntosh \cite{Gordon-McIntosh} proved that
\begin{align}
U_0^{(8)}(q)=\frac{(-q;q^2)_\infty}{(q^2;q^2)_\infty}\sum_{n=-\infty}^\infty \frac{1+q^{2n}}{1+q^{4n}}(-1)^nq^{2n^2+n}.
\end{align}
Note that
\begin{align}
\sum_{n=-\infty}^\infty \frac{1+q^{2n}}{1+q^{4n}}(-1)^nq^{2n^2+n} =2\sum_{n=-\infty}^\infty \frac{q^{2n^2+n}}{1+q^{4n}}.
\end{align}
We have
\begin{align}
U_0^{(8)}(q)\equiv \frac{(-q;q^2)_\infty}{(q^2;q^2)_\infty}\equiv \frac{1}{(q;q)_\infty^3} \pmod{2},
\end{align}
from which we get the desired conclusion immediately.

Since
\begin{align}
U_0^{(8)}(q)=S_0^{(8)}(q^2)+qS_1^{(8)}(q^2).
\end{align}
We get the parity results for $S_0^{(8)}(q)$ and $S_1^{(8)}(q)$ as by-products.
\end{proof}

\begin{theorem}\label{thm-mock-8-U1}
The coefficient $c(U_1^{(8)};n)$ is odd if and only if $8n-1=p^{4a+1}m^2$ for some prime $p$ and integer $m$ with $p\nmid m$. We have
\begin{align}
\#\left\{n\leq N:c(U_1^{(8)};n)\equiv 1 \,\, (\mathrm{mod \,\, 2}) \right\} =\frac{\pi^2}{4} \frac{N}{\log N}+O\left(\frac{N}{\log^2 N} \right).
\end{align}
\end{theorem}
\begin{proof}
From the relation
\begin{align*}
U_1^{(8)}(q)=T_0(q^2)+qT_1(q^2)
\end{align*}
we deduce that
\begin{align}
c(U_1^{(8)};2n)= c(T_0^{(8)};n), \quad c(U_1^{(8)};2n+1) = c(T_1^{(8)};n).
\end{align}
The assertions then follow from Theorems \ref{thm-mock-8-T0} and \ref{thm-mock-8-T1}.
\end{proof}

\section{Mock theta functions of order 10}\label{sec-mock-10}
In his lost notebook \cite{lostnotebook}, Ramanujan recorded four mock theta functions of order 10:
\begin{align*}
&\phi^{(10)}(q):=\sum_{n=0}^\infty \frac{q^{n(n+1)/2}}{(q;q^2)_{n+1}}, 
\quad \psi^{(10)}(q):=\sum_{n=1}^\infty \frac{q^{n(n+1)/2}}{(q;q^2)_n},\\ 
& X^{(10)}(q):=\sum_{n=0}^\infty \frac{(-1)^nq^{n^2}}{(-q;q)_{2n}}, 
\quad \chi^{(10)}(q):=\sum_{n=1}^\infty \frac{(-1)^{n-1}q^{n^2}}{(-q;q)_{2n-1}}. 
\end{align*}
\begin{theorem}
The coefficient $c(\phi^{(10)};n)$ is odd if and only if $n=5k^2+2k$ for some integer $k$. The coefficient $c(\psi^{(10)};n)$ is odd if and only if $n=5k^2+4k+1$ for some integer $k$.
\end{theorem}
\begin{proof}
Choi \cite[Eqs.\ (2.15) and (2.18)]{Choi-1} found  the following Hecke-type series representations:
\begin{align}
\phi^{(10)}(q)=&\frac{(q^2;q^2)_\infty}{(q;q)_\infty^2}\left( \sum_{n=0}^\infty \sum_{|j|\leq n}q^{5n^2+2n-j^2}(1-q^{6n+3}) \right. \nonumber \\
&\left. -2\sum_{n=0}^\infty \sum_{j=0}^n q^{5n^2+7n+2-j^2-j}(1-q^{6n+6}) \right), \label{mock-10-phi}\\
\psi^{(10)}(q)=&\frac{(q^2;q^2)_\infty}{(q;q)_\infty^2}\left(\sum_{n=0}^\infty
\sum_{|j|\leq n}q^{5n^2+4n+1-j^2}(1-q^{2n+1}) \right. \nonumber \\
 & \left. -2\sum_{n=0}^\infty\sum_{j=0}^nq^{5n^2+9n+4-j^2-j}(1-q^{2n+2}) \right). \label{mock-10-psi}
\end{align}
From \eqref{mock-10-phi} and \eqref{mock-10-psi} we deduce that
\begin{align}
\phi^{(10)}(q)\equiv \sum_{n=0}^\infty q^{5n^2+2n}(1+q^{6n+3})\equiv \sum_{n=-\infty}^\infty q^{5n^2+2n} \pmod{2}, \label{10-phi} \\
\psi^{(10)}(q) \equiv \sum_{n=0}^\infty q^{5n^2+4n+1}(1+q^{2n+1}) \equiv \sum_{n=-\infty}^\infty q^{5n^2+4n+1} \pmod{2}. \label{10-psi}
\end{align}
The desired assertions follow immediately from \eqref{10-phi} and \eqref{10-psi}.
\end{proof}

\begin{theorem}\label{thm-mock-10-Xchi}
We have
\begin{align}
X^{(10)}(q)\equiv \frac{(q^8,q^{12},q^{20};q^{20})_\infty}{(q;q)_\infty^3} \pmod{2}, \\
\chi^{(10)}(q) \equiv q\frac{(q^4,q^{16},q^{20};q^{20})_\infty}{(q;q)_\infty^3} \pmod{2}.
\end{align}
\end{theorem}
\begin{proof}
Choi \cite[Eqs.\ (2.2.6),(2.2.8)]{Choi-2} found the following Hecke-type series representations:
\begin{align}
X^{(10)}(q)=&\frac{(q;q)_\infty}{(q^2;q^2)_\infty^2} \left( \sum_{n=0}^\infty \sum_{|j|\leq n}q^{10n^2+2n-2j^2}(1-q^{16n+8}) \right. \nonumber \\
&\left. +2\sum_{n=0}^\infty \sum_{j=0}^n q^{10n^2+12n+3-2j^2-2j}(1-q^{16n+16}) \right), \\
\chi^{(10)}(q)=& \frac{(q;q)_\infty}{(q^2;q^2)_\infty^2}\left(\sum_{n=0}^\infty\sum_{|j|\leq n} q^{10n^2+6n+1-2j^2}(1-q^{8n+4}) \right.\nonumber \\
& \left.+2\sum_{n=0}^\infty \sum_{j=0}^nq^{10n^2+16n+6-2j^2-2j}(1-q^{8n+8})  \right).
\end{align}
We deduce that
\begin{align*}
X^{(10)}(q)\equiv \frac{1}{(q;q)_\infty^3} \sum_{n=-\infty}^\infty q^{10n^2+2n} \equiv \frac{(q^8,q^{12},q^{20};q^{20})_\infty}{(q;q)_\infty^3} \pmod{2}, \\
\chi^{(10)}(q)\equiv \frac{1}{(q;q)_\infty^3}\sum_{n=-\infty}^\infty q^{10n^2+6n+1} \equiv q\frac{(q^4,q^{16},q^{20};q^{20})_\infty}{(q;q)_\infty^3} \pmod{2}.
\end{align*}
This proves the theorem.
\end{proof}

\section{Concluding Remarks}\label{sec-concluding}
If Conjecture \ref{conj-mock} is true, then together with Theorem \ref{thm-main}, we know that the 44 classical mock theta functions can be classified into three classes. The first class consists of 21 functions of parity type $(1,0)$, the second class consists of 19 functions of parity type $(\frac{1}{2},\frac{1}{2})$, and the third class contains 4 functions of type $(\frac{3}{4},\frac{1}{4})$.

Now we briefly discuss the 23 functions listed in Conjecture \ref{conj-mock}. Note that $c(g;n)\equiv p(n)$ (mod 2) for $g$ being $f^{(3)}(q)$, $\phi^{(3)}(q)$ or $\phi^{(6)}(q)$. Furthermore, we know from Theorem \ref{thm-mock-6-mu} that $c(2\mu^{(6)};2n)\equiv p(n)$ (mod 2). Thus the parity types of these functions are determined by the parity type of $p(n)$.  Next, we have $c(g;n)\equiv p_{-3}(n)$ (mod 2) for $g$ being $\mu^{(2)}(q)$ and $U_0^{(8)}(q)$.  Theorem \ref{thm-mock-8-U0S0S1} tells us that
\begin{align*}
c(S_0^{(8)};n)\equiv p_{-3}(2n) \pmod{2} \quad \text{and} \quad c(S_1^{(8)};n)\equiv p_{-3}(2n+1) \pmod{2}.
\end{align*}
Thus if we can verify the conjecture for $\mu^{(2)}(q)$ and $S_0^{(8)}(q)$, then it also  holds for $U_0^{(8)}(q)$ and $S_1^{(8)}(q)$.

Furthermore, from \eqref{5-chi-phi} we see that if we can prove the conjecture for $\phi_0^{(5)}(q)$ and $\phi_1^{(5)}(q)$, then it also holds for $\chi_0^{(5)}(q)$ and $\chi_1^{(5)}(q)$.

Finally, from Theorem \ref{thm-mock-6-lambda} we know that if the conjecture holds for $\psi^{(6)}(q)$, then it also holds for $\lambda^{(6)}(q)$.

From the above, we know that we only need to verify Conjecture \ref{conj-mock} for 15 functions:
\begin{align}\label{15-functions}
\mu^{(2)}(q), f^{(3)}(q), \chi^{(3)}(q), \phi_0^{(5)}(q), \phi_1^{(5)}(q), \psi^{(6)}(q), \gamma^{(6)}(q), \mathcal{F}_0^{(7)}(q), \nonumber \\ \mathcal{F}_1^{(7)}(q), \mathcal{F}_2^{(7)}(q), S_0^{(8)}(q), X^{(10)}(q), \chi^{(10)}(q), f_0^{(5)}(q), f_1^{(5)}(q).
\end{align}

For any integer sequence $\{c(n):n\geq 0\}$, we denote
\begin{align}
\delta(\{c(n)\};X):=\frac{1}{X}\#\left\{0\leq n<X: c(n) \text{ is odd}\right\}.
\end{align}
Note that for $f_0^{(5)}(q)$ and $f_1^{(5)}(q)$, from Theorem \ref{thm-mock-5-f0-f1}  we only need to verify that the sequences $c(f_0^{(5)};2n)$ and $c(f_1^{(5)};2n+1)$ are both of type $(\frac{1}{2},\frac{1}{2})$. Taking $X=100000$, using Maple we get the values of $\delta(\{c(n)\};X)$ for the 15 functions listed in \eqref{15-functions}. See Table \ref{tab-density}. These data give numerical evidence that support the truth of Conjecture \ref{conj-mock}. However, it might be quite challenging to prove the conjecture.
\begin{table}[h]
\centering
\renewcommand\arraystretch{1.5}
\begin{tabular}{cccccccc}
\specialrule{1pt}{1pt}{1pt}
  $c(n)$ & $c(\mu^{(2)};n)$ & $c(f^{(3)};n)$ & $c(\chi^{(3)};n)$ & $c(\phi_0^{(5)};n)$ & $c(\phi_1^{(5)};n)$  \\
  \hline
  $\delta(\{c(n)\};X)$ & 0.50161 & 0.50201 & 0.49847 & 0.50226  &0.50162   \\
  \hline
  $c(n)$ & $c(\psi^{(6)};n)$ & $c(\gamma^{(6)};n)$ &  $c(\mathcal{F}_0^{(7)};n)$ & $c(\mathcal{F}_1^{(7)};n)$ & $c(\mathcal{F}_2^{(7)};n)$ \\
  \hline
  $\delta(\{c(n)\};X)$ & 0.50086 & 0.49847 & 0.49857 & 0.49667 & 0.50102     \\
  \hline
  $c(n)$  &  $c(S_0^{(8)};n)$  & $c(X^{(10)};n)$ & $c(\chi^{(10)};n)$ & $c(f_0^{(5)};2n)$ & $c(f_1^{(5)};2n+1)$  \\
 $\delta(\{c(n)\};X)$  & 0.50041 & 0.50063 & 0.50244 & 0.50188 & 0.49838 & \\
 \specialrule{1pt}{1pt}{1pt}
\end{tabular}
\caption{Values of $\delta(\{c(n)\};X)$ with $X=100000$}\label{tab-density}
\end{table}

\subsection*{Acknowledgements}
The author was supported by the National Natural Science Foundation of China (11801424) and a start-up research grant of the Wuhan University.


\begin{thebibliography}{0}




\bibitem{Andrews-TAMS} G.E. Andrews, The fifth and seventh order mock theta functions, Trans. Amer. Math.
Soc. 293 (1) (1986), 113--134.



\bibitem{lost-notebook5} G.E. Andrews and B.C. Berndt, Ramanujan's Lost Notebook, Part V, Springer, New York, 2018.


\bibitem{ADH} G.E. Andrews, F. Dyson and D. Hickerson, Partitions and indefinite quadratic forms, Invent. Math. 91 (1988), 391--407.


\bibitem{AGL} G.E. Andrews, F.G. Garvan and J. Liang, Self-conjugate vector partitions and the parity of the spt-function, Acta Arith. 158.3 (2013), 199--218.

\bibitem{Andrews-Hickerson} G.E. Andrews and D. Hickerson, Ramanujan's ``Lost'' Notebook \uppercase\expandafter{\romannumeral7}:
The sixth order mock theta functions, Adv. Math. 89 (1991), 60--105.


\bibitem{APSY} G.E. Andrews, D. Passary, J. Sellers and A.J. Yee, Congruences related to the Ramanujan/Watson mock theta functions $\omega(q)$ and $\nu(q)$, Ramanujan J 43 (2017), 347--357.




\bibitem{BGS} J.\  Bella\"iche, B. Green and K. Soundararajan, Nonzero coefficients of half-integral weight modular forms mod $\ell$, Res. Math. Sci. 5 (2018) 6.

\bibitem{Nicolas} J.\  Bella\"iche and J.L. Nicolas, Parit\'e des coefficients de formes modulaires, Ramanujan J. 40 (2016), 1--44.

\bibitem{BCGKMW} J.\ Berg, A. Castillo, R. Grizzard, V. Kala, R. Moy and C. Wang, Congruences for Ramanujan's $f$ and $\omega$ functions via generalized Bocherds products, Ramanujan J. 35 (2014), 327--338.

\bibitem{Berndt-book} B.C. Berndt, Number Theory in the Spirit of Ramanujan, AMS, 2016.


\bibitem{Berndt-Chan} B.C. Berndt and S.H. Chan, Sixth order mock theta functions, Adv. Math. 216 (2007), 771--786.


\bibitem{BYZ} B.C. Berndt, A.J. Yee and A. Zaharescu, On the parity of partition functions, Int. J. Number Theory 14(4) (2003), 437--459.



\bibitem{BS-book} Z.I. Borevich and I.R. Shafarevich, Number Theory, Academic Press, 1966.


\bibitem{BSS} E.H.M. Brietzke, R. Silva and J. A. Sellers, Congruences related to an eighth order mock theta function of Gordon and McIntosh, J. Math. Anal. Appl. 479 (2019), 62--89.


\bibitem{BO} J.H. Bruinier  and· K. Ono, Identities and congruences for Ramanujan's $\omega(q)$, Ramanujan J. 23 (2010), 151--157.


\bibitem{Chan-Mao} S. H. Chan and R. Mao, Two congruences for Appell–Lerch sums, Int. J. Number Theory 8(1) (2012), 111--123.


\bibitem{Chen-Wang} D. Chen and L. Wang, Representations of mock theta functions, Adv. Math. 365 (2020), 107037.

\bibitem{Chen} S.-C. Chen, Odd values of the Rogers-Ramanujan functions, C. R. Acad. Sci. Paris, Ser. I 356 (2018), 1081--1084.

\bibitem{Chern-Wang} S. Chern and C. Wang, An infinite family of congruences arising from a second order mock theta function, arXiv:1803.01976.

\bibitem{Choi-1}  Y.-S. Choi, Tenth order mock theta functions in Ramanujan's lost notebbook, Invent. Math. 136 (1999), 497--569.

\bibitem{Choi-2} Y.-S. Choi, Tenth order mock theta functions in Ramanujan's lost notebbook II, Adv. Math. 156 (2000), 180--285.




\bibitem{CGH} S.-P. Cui, N.S.S. Gu, and L.-J. Hao, On second and eighth order mock theta functions, Ramanujan J. 50 (2019), 393--422.





\bibitem{Garthwaite} S. Garthwaite, The coefficients of the $\omega(q)$ mock theta function, Int. J. Number Theory 4(6) (2008), 1027--1042.


\bibitem{GP} S. Garthwaite and D. Penniston,  $p$-adic properties of Maass forms arising from theta series, Math. Res. Lett. 15  (2008), 459--470.

\bibitem{Garvan-2015} F.G. Garvan, Universal mock theta functions and two-variable Hecke-Rogers identities, Ramanujan J. 36 (2015), 267--296.

\bibitem{Garvan-arXiv} F.G. Garvan, New fifth and seventh order mock theta function identities, preprint, arXiv 1907.04803v1.

\bibitem{t-core} F. Garvan, D. Kim and D. Stanton, Cranks and $t$-cores, Invent. Math. 101 (1990), 1--17.

\bibitem{Gasper-Rahman} G. Gasper and M. Rahman, Basic Hypergeometric Series, Encyclopedia of Mathematics and Its Applications, 2nd edn., vol. 35. Cambridge Univ. Press, Cambridge, 2004.

\bibitem{Gordon} B. Gordon, On the parity of the Rogers-Ramanujan coefficients, in Partitions, $q$-Series and Modular Forms, K. Alladi and F. Garvan, eds., Develop. in Math. 23, 2011, Springer, New York, pp. 83--93.

\bibitem{Gordon-McIntosh} B. Gordon and R.J. McIntosh, Some eighth order mock theta functions, J. London Math. Soc. (2) 62 (2000), 321--335.

\bibitem{Gordon-McIntosh-Survey} B. Gordon and R.J. McIntosh, A survey of classical mock theta functions, in Partitions, $q$-Series and Modular Forms, K. Alladi and F. Garvan, eds., Develop. in Math. 23, 2011, Springer, New York, pp. 95--144.






\bibitem{Hardy-Wright-book} G.H. Hardy and E.M. Wright, An Introduction to the Theory of Numbers (6th ed.), Oxford University Press (2006).


\bibitem{Hickerson} D. Hickerson, On the seventh order mock theta functions, Invent. Math. 94 (1988), 661--677.


\bibitem{Hickerson-Mortenson} D. Hickerson and E. Mortenson, Hecke-type double sums, Appell-Lerch sums and mock theta functions, I, Proc. London Math. Soc. (3) 109 (2014), 382--422.

\bibitem{HGB} M.D. Hirschhorn, F. Garvan and J. Borwein, Cubic analogs of the Jacobian cubic theta function $\theta(z,q)$, Canad. J. Math. 45 (1993), 673--694.

\bibitem{HS} M.D. Hirschhorn and J.A. Sellers, Elementary proofs of various facts about 3-cores, Bull. Aust. Math. Soc. 79(3) (2009), 507--512.


\bibitem{Kolberg} O. Kolberg, Note on the parity of the partition function, Math. Scand. 7 (1959), 377--378.

\bibitem{Lin} B.L.S. Lin, Overpartitions related to the mock theta function $V_0(q)$, Bull. Aust. Math. Soc. doi:10.1017/S0004972719001618.



%
%
\bibitem{Liu2013IJNT} Z.-G. Liu, On the $q$-derivative and $q$-series expansions, Int. J. Number Theory  9(8) (2013), 2069--2089.

\bibitem{Lovejoy} J. Lovejoy, Overpartitions and real quadratic fields, J. Number Theory 106 (2004), 178--186.

\bibitem{Mao} R. Mao, Two identities on the mock theta function $V_0(q)$, J. Math. Anal. Appl. 479 (2019), 122--134.

\bibitem{Mao-BAMS} R. Mao, Arithmetic properties of coefficients of the mock theta function $B(q)$, Bull. Aust. Math. Soc. doi:10.1017/S0004972719001175.

\bibitem{McIntosh-2007Canad} R.J. McIntosh, Second order mock theta functions, Canad. Math. Bull. 50(2) (2007), 284--290.




\bibitem{Mirsky} L. Mirsky, The distribution of values of the partition function in residue classes, J. Math. Anal. Appl. 93 (1983), 593--598.


\bibitem{Mortenson-2013} E.T. Mortenson, On three third order mock theta functions and Hecke-type double sums, Ramanujan J. 30 (2013), 279--308.

%



\bibitem{NRS} J.-L. Nicolas, I.Z. Ruzsa and A. S\'ark\" ozy, On the parity of additive representation functions. With an appendix by J.-P. Serre, J. Number Theory 73 (1998), 292--317.

\bibitem{Parkin-Shanks} T.R. Parkin and D. Shanks, On the distribution of partity in the partition funciton, Math. Comp. 21 (1967), 466--480.

\bibitem{Qu-Wang-Yao} Y. K. Qu, Y. J. Wang and O. X. M. Yao, Generalizations of some conjectures of Chan on
congruences for Appell–Lerch sums, J. Math. Anal. Appl. 460(1) (2018), 232--238.


\bibitem{lostnotebook}
S. Ramanujan, The Lost Notebook and Other Unpublished Papers, Narosa, New Delhi, 1988.



\bibitem{Sally} J.D. Sally and P.J. Sally, Jr., Roots to Research: a vertical development of mathematical problems, American Mathematical Society, 2007.

\bibitem{Srivastava} B. Srivastava, Hecke modular form expansions for eighth order mock theta functions, Tokyo J. Math. 28(2) (2015), 563--577.




\bibitem{Waldherr} M. Waldherr, On certain explicit congruences for mock theta functions, Proc. Amer. Math. Soc. 139 (3) (2011), 865--879.

\bibitem{Wang2017} L. Wang, New congruences for partitions related to mock theta functions, J. Number Theory  175 (2017), 51--65.

\bibitem{Watson} G.N. Watson, The final problem: an account of the mock theta functions, J. London Math. Soc. 11 (1936), 55--80.

\bibitem{Watson-2} G.N. Watson, The mock theta functions (2), Proc. London Math. Soc. (2) 42 (1937), 274--304.

\bibitem{Xia} E.X.W. Xia, Arithmetic properties for a partition function related to the Ramanujan/Watson mock theta function $\omega(q)$, Ramanujan J. 46 (2018), 545--562.

\bibitem{Xia-Yao} E.X.W. Xia and O.X.M. Yao, Analogues of Ramanujan's partition identities, Ramanujan J. 31 (2013), 373--396.


\bibitem{Zwegers} S.P. Zwegers, Mock theta functions, PhD Thesis, Utrecht PhD thesis (2002), ISBN 90-393-3155-3.

\bibitem{Zwegers-Rama} S.P. Zwegers, On two fifth order mock theta functions, Ramanujan J. 20 (2009), 207--214.
\end{thebibliography}
\end{document}